\documentclass[letterpaper, 11pt]{article}
\usepackage[utf8]{inputenc}
\usepackage[T1]{fontenc}
\usepackage{lmodern}
\usepackage{graphicx,amsmath,amsfonts,latexsym,amsthm,amssymb,mathrsfs,mathtools,stmaryrd}
\usepackage[abs]{overpic}
\usepackage[dvipsnames]{xcolor}
\usepackage[linktocpage=true]{hyperref}
\usepackage{caption}
\usepackage{dsfont}
\usepackage{tikz-cd}
\usepackage{comment}
\usepackage{enumitem}
\usepackage{microtype}
\numberwithin{equation}{section}

\usepackage{graphpap}
\usepackage{svg}
\tikzset{
  symbol/.style={
    draw=none,
    every to/.append style={
      edge node={node [sloped, allow upside down, auto=false]{$#1$}}}
  }
}

\definecolor{green1}{rgb}{0.,0.4,0.}
\definecolor{blue1}{rgb}{0.,0.,0.8}
\definecolor{red1}{rgb}{0.8,0.,0.}

\usepackage[margin=1.25in]{geometry}

\usepackage[backend=biber, maxnames=4, style=alphabetic, sorting=nyt]{biblatex}
\hypersetup{
  colorlinks,
  citecolor=teal,
  linkcolor=purple,
  urlcolor=teal}
  
\newtheorem{thm}{Theorem}[section]
\newtheorem{cor}[thm]{Corollary}
\newtheorem{lem}[thm]{Lemma}
\newtheorem{question}[thm]{Question}
\newtheorem{prop}[thm]{Proposition}
\newtheorem{conj}[thm]{Conjecture}

\newtheorem{defn}[thm]{Definition}
\newtheorem{ex}[thm]{Example}
\newtheorem{rem}[thm]{Remark}

\newtheorem*{thm*}{Theorem}
\newtheorem*{cor*}{Corollary}

\newcommand{\Z}{\mathbb{Z}}
\newcommand{\R}{\mathbb{R}}

\title{A symplectic viewpoint on Anosov flows}

\author{Thomas Massoni\thanks{Department of Mathematics, MIT, Cambridge, USA. Email address: \href{mailto:thomas10@mit.edu}{thomas10@mit.edu}}}

\date{\today}

\begin{document}

\maketitle

\begin{abstract}
This survey explores the geometry of three-dimensional Anosov flows from the perspective of contact and symplectic geometry, following the work of Mitsumatsu, Eliashberg--Thurston, Hozoori, and the author. We also present a few original results and discuss various open questions and conjectures.
\end{abstract}

\tableofcontents

\section{Introduction}

Anosov flows are fundamental objects in smooth hyperbolic dynamics. They were originally introduced by Anosov~\cite{A63, A67} as a generalization of geodesic flows on hyperbolic manifolds. The study of Anosov flows in dimension three is particularly rich, as their dynamical and topological properties exhibit deep and intricate interactions. For an overview of classical results and references on Anosov flows, see the nice survey of Barthelm\'{e}~\cite{B17}. A more comprehensive exposition can be found in the monograph of Fisher and Hasselblatt~\cite{FH19}. 

A significant connection to contact and symplectic geometry was obtained independently by Mitsumatsu~\cite{M95} and Eliashberg and Thurston~\cite{ET}, who introduced the concept of \emph{projectively/conformally Anosov flow} on $3$-manifolds. They established a correspondence between such flows and \emph{bicontact structures}, defined as transverse pairs of contact structures with opposite orientations. More recently, Hozoori~\cite{Hoz24} strengthened this connection by showing that (oriented) Anosov flows can be completely characterized in terms of bicontact geometry. Building on these results, the author refined the correspondence in~\cite{Mas23}, proving that the space of Anosov flows is \emph{homotopy equivalent} to a space of suitable bicontact structures. Some aspects of this correspondence can be generalized to the study of foliations~\cite{Mas24}. 

A remarkable feature of Mitsumatsu’s construction is that it naturally gives rise to a \emph{Liouville structure} on the `thickening' $[-1,1] \times M$ of the $3$-manifold $M$ on which the Anosov flow lives.  This Liouville structure is \emph{not Weinstein} but possesses a rich geometry mirroring the dynamics of the Anosov flow. The symplectic invariants of such non-Weinstein Liouville domains were recently investigated by the author in joint work with Cieliebak, Lazarev, and Moreno~\cite{CLMM}. In work in preparation with Bowden~\cite{BM}, we further show that these Liouville structures are preserved under \emph{orbit equivalence} of Anosov flows, and so are their symplectic and Floer-theoretical invariants.

\medskip

This article surveys the geometric interactions between (projectively) Anosov flows and bicontact structures, focusing on the work of Mitsumatsu, Eliashberg--Thurston, Hozoori, and the author. This article also contains two original contributions: Theorem~\ref{thm:volumepres} which provides a symplectic characterization of volume preserving Anosov flows generalizing~\cite[Theorem 3.15]{Mas23}, and Section~\ref{sec:gAL} on a natural generalization of the notion of Anosov Liouville structure. We also discuss various open questions and conjectures.

\medskip

\begin{center}
\emph{Throughout this article, $M$ denotes a connected, closed, oriented, smooth $3$-manifold.}
\end{center}

\medskip

\paragraph{Acknowledgments.} I would like to warmly thank the organizers and participants of the workshop \emph{Symplectic Geometry and Anosov Flows} held in Heidelberg in July 2024, for creating such a stimulating and inspiring atmosphere. I am especially grateful to Surena Hozoori, Jonathan Bowden, and Jonathan Zung, for teaching me most of what I know about Anosov flows. I am grateful to the anonymous referee for their comments and suggestions.

\section{Anosov and projectively Anosov flows}

In this section, we recall some basic definitions for Anosov flows and their associated foliations. We also discuss some weaker notions (projectively Anosov, semi-Anosov). For the purpose of constructing contact and Liouville structures, it will be convenient to rephrase those definitions in terms of suitable $1$-forms.

        \subsection{Definitions}

Let $\Phi = \{\phi^t\}_{t \in \R}$ be a flow on $M$ generated by a nonsingular $\mathcal{C}^1$ vector field $X$.

\begin{defn} \label{def:anosov}
$\Phi$ is \textbf{Anosov} if there exists a continuous invariant \textbf{hyperbolic splitting}
\begin{align} \label{anosovsplit}
TM = \langle X \rangle \oplus E^s \oplus E^u
\end{align}
where $E^s, E^u$ are $1$-dimensional bundles such that for some (any) Riemannian metric $g$ on $M$, there exist constants $C, a >0$ such that for all $v \in E^s$ and $t \geq 0$, $$\Vert d\phi^t(v)\Vert \leq C e^{-at} \Vert v \Vert,$$
and for all $v \in E^u$ and $t \geq 0$, $$\Vert d\phi^t(v)\Vert \geq C e^{at} \Vert v \Vert.$$
The line bundles $E^s$ and $E^u$ are called the (strong) stable and unstable bundles of $\Phi$, respectively.
\end{defn}

A weaker notion was introduced by Mitsumatsu~\cite{M95} under the name of \emph{projectively} Anosov flow, and independently by Eliashberg--Thurston~\cite{ET} under the name of \emph{conformally} Anosov flow:

\begin{defn} \label{def:projanosov} 
$\Phi$ is \textbf{projectively Anosov} if there exists a continuous invariant splitting
\begin{align} \label{projanosovsplit}
TM \slash \langle X \rangle = N_X = \overline{E}^s \oplus \overline{E}^u
\end{align}
where $\overline{E}^s, \overline{E}^u$ are $1$-dimensional bundles such that for some (any) Riemannian metric $\overline{g}$ on $N_X$, there exist constants $C, a >0$ such that for all unit vectors $v_s \in \overline{E}^s, v_u \in \overline{E}^u$, and all $t \geq 0$, $$ \Vert d\phi^t(v_u)\Vert \geq C e^{at} \, \Vert d\phi^t(v_s) \Vert.$$
Such a splitting is called a \textbf{dominated splitting}. We denote by $E^{ws} \coloneqq \pi^{-1}(\overline{E}^s)$ and $E^{wu} \coloneqq \pi^{-1}(\overline{E}^u)$ the weak stable and weak unstable bundles of $\Phi$, respectively.
\end{defn}

It is easy to see that Anosov flows are projectively Anosov, with dominated splitting $N_X = \pi(E^s) \oplus \pi(E^u)$. From the perspective of symplectic geometry, it will be relevant to introduce another type of flows whose behavior interpolates between projectively Anosov and Anosov flows. Those were already mentioned in~\cite{Mas24}, and in~\cite{Hoz24-reg} under the name of \emph{partially hyperbolic} flows.

\begin{defn}
    $\Phi$ is \textbf{semi-Anosov} if it is projectively Anosov, and there exists a continuous invariant splitting
    \begin{align}
    TM = E^{ws} \oplus E^u \label{seminanosovsplit}
    \end{align}
    where $E^{ws}$ is as in Definition~\ref{def:projanosov}, and $E^u$ is a $1$-dimensional bundle, such that for some (any) Riemannian metric $g$ on $M$, there exist constants $C,a > 0$ such that for all $v \in E^u$ and $t \geq 0$,
    $$\Vert d\phi^t(v)\Vert \geq C e^{at} \Vert v \Vert.$$
    We call $E^u$ the (strong) unstable bundle of $\Phi$.
\end{defn}

In other words, a semi-Anosov flow is a projectively Anosov flows that admits a \emph{strong} unstable bundle, but not necessarily a strong stable one.

In all those definitions, if the constant $C$ can be chosen to be $1$, the corresponding metrics $g$ and $\overline{g}$ are called \textbf{adapted} to $\Phi$. Is is well-known that every three-dimensional (projectively) Anosov flow admits a smooth adapted metric, and the same holds for semi-Anosov flows (see~\cite[Proposition 5.1.5]{FH19} for a proof).

An Anosov (resp.~projectively Anosov) flow $\Phi$ is \textbf{oriented} if $E^s$ and $E^u$ are oriented (resp.~$\overline{E}^s$ and $\overline{E}^u$ are oriented) and their orientations are compatible with the splitting~\eqref{anosovsplit} (resp.~the splitting~\eqref{projanosovsplit}). A semi-Anosov flow $\Phi$ is oriented if it is oriented as a projectively Anosov flow; a natural orientation on $E^u$ is induced by~\eqref{seminanosovsplit}.

The following lemma is an easy consequence of the definition and the independence on the choice of metrics. The proof is left as an exercise to the reader.

\begin{lem} \label{lem:semianosov}
    A projectively Anosov flow $\Phi$ is Anosov if and only if $\Phi$ and $\Phi^{-1}$ are semi-Anosov.
\end{lem}

\begin{center}
\noindent\fbox{%
\parbox{\linewidth - 3\fboxsep}{%
\textbf{From now on, we assume that all of the Anosov and projectively Anosov flows under consideration are \underline{smooth} and \underline{oriented}.}}}
\end{center}

Not every (projectively) Anosov flow is oriented (or rather orientable), but it becomes orientable after passing to a suitable double cover.

        \subsection{Basic properties}

\paragraph{First examples.} Le us briefly recall the most elementary examples of Anosov flows in dimension three.
\begin{itemize}[leftmargin=*]
    \item \textit{Suspension of linear hyperbolic diffeomorphism of the torus.} Let $A \in \mathrm{SL}(2, \Z)$ be hyperbolic, i.e., $\vert\mathrm{tr}(A) \vert> 2$. It naturally induces a diffeomorphism of $\mathbb{T}^2 \coloneqq \R^2 \slash \Z^2$, still denoted by $A$, which has two invariant linear foliations parallel to the eigenvectors of $A$. Note that the two (real) eigenvalues of $A$ are of the form $e^\nu$ and $e^{-\nu}$ if $\mathrm{tr}(A) > 0$, and $-e^\nu$ and $-e^{-\nu}$ if $\mathrm{tr}(A) < 0$, for some $\nu \in \R_{> 0}$. On the suspension 
    $$M_A \coloneqq [0,1]_t \times \mathbb{T}^2 \slash_{(1,x) \sim (0,Ax)}$$
    of $A$, the suspension flow $\Phi$ generated by $\partial_t$ has first return map $A$. It is easy to check that it is Anosov, and it is orientable exactly when $\mathrm{tr}(A) > 0$.
    
    \item \textit{Geodesic flows on hyperbolic surfaces.} Let $(\Sigma, g)$ be a closed oriented surface endowed with a hyperbolic metric $g$. Then, the geodesic flow of $g$, viewed as a flow on the unit tangent bundle $UT\Sigma$ of $\Sigma$, is Anosov and orientable. This can be seen by realizing $UT\Sigma$ as a quotient of $U \mathbb{H}$, the unit tangent bundle of the hyperbolic half-plane, by a subgroup of $\mathrm{SL}(2, \R)$. There is a well-known equivariant diffeomorphism $U \mathbb{H} \cong \mathrm{PSL}(2, \R)$ under which the geodesic flow is generated by the matrix
    $$\begin{pmatrix}
        \frac{1}{2} & 0 \\
        0 & -\frac{1}{2}
    \end{pmatrix} \in \mathfrak{sl}(2, \R),$$
    viewed as an invariant vector field on $\mathrm{PSL}(2, \R)$. The strong stable and unstable bundles can also be described explicitly in terms of matrices in $\mathfrak{sl}(2, \R)$, see~\cite{Gei}.
\end{itemize}

Note that in both examples, the Anosov splitting is \emph{smooth}. This is a special feature of \textbf{algebraic} Anosov flows, which exactly correspond to (manifold quotients of) the two previous examples in dimension three.

\paragraph{Structural stability.} 

Recall the natural equivalence relation on flows.

\begin{defn} Two flows $\Phi = \{ \phi^t\}$ and $\Psi= \{ \psi^t\}$ on $M$ are \textbf{topologically equivalent}, or \textbf{orbit equivalent}, if there exist a homeomorphism $h: M \rightarrow M$ and a continuous map $\tau : \R \times M \rightarrow \R$ such that $\tau(t,p) \geq 0$ for $t \geq 0$, and $$\psi^{\tau(t,p)} = h \circ \phi^t \circ h^{-1}(p)$$ for all $t \in \R$ and $p \in M$.
\end{defn}

In other words, the topological equivalence $h$ sends the oriented trajectories of $\Phi$ onto the oriented trajectories of $\Psi$, but does not necessarily preserve their parametrization. 

A key feature of Anosov flows is their \textbf{structural stability}:

\begin{thm*}
    Let $\Phi$ be an Anosov flow generated by a $\mathcal{C}^1$ vector field $X$. There exists a $\mathcal{C}^1$-neighborhood $\mathcal{U}$ of $X$ in the space of $\mathcal{C}^1$ vector field such that every $X' \in \mathcal{U}$ generates an Anosov flow which is orbit equivalent to $\Phi$.
\end{thm*}

This was originally proved by Anosov~\cite{A63, A67}; see also~\cite{FH19} for a modern treatment. This result allows us to restrict our attention to smooth Anosov flows, since any Anosov flow is orbit equivalent to a nearby smooth one. We note that there is no similar statement for projectively Anosov or semi-Anosov flows.

\paragraph{Invariant foliations.}

Although the invariant splitting in Definition~\ref{def:anosov} is only continuous, the hyperbolicity guarantees that $E^s$ and $E^u$ are uniquely integrable and tangent to $1$-dimensional foliations---a continuous vector field is not necessarily tangent to a foliation! As a result, the weak stable and unstable bundles are also uniquely integrable and tangent to $2$-dimensional foliations, denoted by $\mathcal{F}^{ws}$ and $\mathcal{F}^{wu}$, respectively.

This can be extended to semi-Anosov flows to show that $E^u$ and $E^{wu}$ are uniquely integrable (it follows from Hadamard--Perron theorem and the Unstable Manifold theorem as in the Anosov case, see~\cite{FH19}). However, it is \emph{not true} that the weak bundles of a \emph{projectively} Anosov flow are uniquely integrable; in~\cite[Example 2.2.9]{ET}, Eliashberg and Thurston construct an example of projectively Anosov flow without invariant foliations. However, using the generalization of Burago and Ivanov's construction of \emph{branching} foliations~\cite{BI08} established in~\cite{Mas24}, one can show that the weak bundles of an oriented projectively Anosov flow are tangent to (somewhat canonical) invariant branching foliations.

An important feature of the weak foliations of a (semi-)Anosov flow is that they are \textbf{taut}. Recall that a foliation is (everywhere) taut if through every point in $M$ passes a closed loop transverse to the foliation. This is equivalent to the existence of a (smooth) closed $2$-form positive on the leaves of the foliation; such `dominating' $2$-forms will be constructed later. An important feature of taut foliations is that transverse loops are homotopically essential (non-contractible) by Novikov's theorem (there exist more straightforward arguments for Anosov foliations).

Finally, let us note that the \emph{weak} foliations of an Anosov flow are preserved under orbit equivalence, in the following sense: if $\Phi$ and $\Psi$ are orbit equivalent via an orbit equivalence $h$, then $h$ sends the weak (un)stable leaves of $\Phi$ onto the weak (un)stable leaves of $\Psi$. A similar statement holds for semi-Anosov flows. However, the orbit equivalence does not preserve the \emph{strong} foliations in general.

\paragraph{Regularity of Anosov bundles.} In general, the strong bundles of an Anosov flow do not integrate to $\mathcal{C}^1$-foliations but to topological foliations with smooth leaves (the leaves themselves only vary continuously). However, if the flow is smooth, then its \emph{weak} foliations are known to be $\mathcal{C}^1$. In dimension three, more is true: the weak bundles themselves are $\mathcal{C}^1$ (see~\cite{H94}). In the semi-Anosov case, $E^{wu}$ is also $\mathcal{C}^1$ (see~\cite[Section 2.3]{Hoz24-reg}). This will allow us to represent these plane fields as kernels of $\mathcal{C}^1$ $1$-forms in the next section. Unfortunately, this does not hold for projectively Anosov flows; see~\cite[Example 2.2.9]{ET}.

In general, the Anosov splitting \emph{cannot} be too regular. For instance, classical results of Ghys imply that if an Anosov flow has a $\mathcal{C}^2$ invariant splitting, then it is orbit equivalent to an algebraic Anosov flow~\cite{Gh92, Gh93}.

        \subsection{Defining pairs}

One can encode the existence of a hyperbolic or dominated splitting for a flow $\Phi$ in terms of the expansion rates of the flow in the (weak) stable and unstable directions. Let us recall~\cite[Definition 3.11]{Hoz24} (see also~\cite[Proposition 3.2]{Mas23}). Let $\Phi$ be a projectively Anosov flow on $M$ generated by a vector field $X$ and $\overline{g}$ be a Riemannian metric on $N_X$. The \textbf{expansion rates} in the stable and unstable directions for $\overline{g}$ are continuous functions $r_s, r_u : M \rightarrow \R$ defined by $$ r_s \coloneqq \left. \frac{\partial}{\partial t}\right\rvert_{t=0} \ln \Vert d\phi^t (\overline{e}_s)\Vert, \qquad r_u \coloneqq \left. \frac{\partial}{\partial t}\right\rvert_{t=0} \ln \Vert d\phi^t (\overline{e}_u)\Vert,$$
where $\overline{e}_s$ and $\overline{e}_u$ are unit sections of $\overline{E}^s$ and $\overline{E}^u$, respectively, which are continuous and continuously differentiable along the flow $\Phi$. It follows from the definition that $$\mathcal{L}_X \overline{e}_s = -r_s \overline{e}_s, \qquad \mathcal{L}_X \overline{e}_u = -r_u \overline{e}_u.$$
The sections $\overline{e}_s$ and $\overline{e}_u$ can be dualized to obtain $1$-forms $\alpha_s$ and $\alpha_u$ on $M$ satisfying:
\begin{align*}
    \overline{\alpha}_s(\overline{e}_s) &= 1, & \ker \alpha_s &= E^{wu}, & \mathcal{L}_X \alpha_s &= r_s \, \alpha_s,\\
    \overline{\alpha}_u(\overline{e}_u) &= 1, & \ker \alpha_u &= E^{ws}, & \mathcal{L}_X \alpha_u &= r_u \, \alpha_u.
\end{align*}
Here, $\overline{\alpha}_s$ and $ \overline{\alpha}_u$ denote the $1$-forms on $N_X$ induced by $\alpha_s$ and $\alpha_u$, respectively. These $1$-forms are not $\mathcal{C}^1$ a priori, but they are continuous and continuously differentiable along $X$.
Then, the existence of a dominated splitting with an adapted metric readily implies:
$$r_s < r_u.$$
If $\Phi$ is semi-Anosov, we further get:
$$ 0 < r_u.$$
Finally, if $\Phi$ is Anosov then the hyperbolicity of the invariant splitting implies:
$$r_s < 0 < r_u.$$

Importantly, this procedure can be reverse, and one can \emph{characterize} (projectively, semi-) Anosov flows in term of suitable such pairs of $1$-forms $(\alpha_s, \alpha_u)$. The following result from~\cite{Mas23} (and~\cite{Mas24} in the semi-Anosov case) was obtained by combining results from~\cite{M95} and~\cite{Hoz24, Hoz24-reg}; see also~\cite[Proposition 2.15]{Hoz24-reg} for a slightly different formulation:

\begin{prop} \label{prop:defpair}
Let $\Phi$ be a smooth (or $\mathcal{C}^1$), nonsingular flow on $M$ generated by a vector field $X$, then
\begin{enumerate}[label=(\arabic*)]
\item $\Phi$ is oriented projectively Anosov if and only if there exist a pair of $1$-forms $(\alpha_s, \alpha_u)$ and continuous functions $r_u, r_s : M \rightarrow \R$ such that
	\begin{align*}
	 \alpha_s(X) &= \alpha_u(X) = 0,  & \overline{\alpha}_s \wedge \overline{\alpha}_u > 0, & & \mathcal{L}_X \alpha_s = r_s \, \alpha_s, & &\mathcal{L}_X \alpha_u = r_u \, \alpha_u,
	\end{align*}
	and $r_s < r_u$.
\item $\Phi$ is oriented semi-Anosov if and only if there exist $\alpha_s$, $\alpha_u$, $r_s$, $r_u$ as above such that $r_s < r_u$ and $0 < r_u$.
\item $\Phi$ is oriented Anosov if and only if there exist $\alpha_s$, $\alpha_u$, $r_s$, $r_u$ as above such that $r_s < 0 < r_u$.
\end{enumerate}
Moreover, $\ker \alpha_u = E^{ws}$ and $\ker \alpha_s = E^{wu}$.
\end{prop}

This motivates the following
\begin{defn} \label{def:defpair}
A \textbf{defining pair} for a projectively Anosov (resp.~semi-Anosov, Anosov) flow $\Phi$ is a pair of $1$-forms $(\alpha_s, \alpha_u)$ as in item (1) (resp.~(2), (3)) of Proposition~\ref{prop:defpair}. In the Anosov case, we further require the pair to be $\mathcal{C}^1$. We call $\alpha_s$ (resp.~$\alpha_u$) a \textbf{$s$-defining $1$-form} (resp.~\textbf{$u$-defining $1$-form}) for the flow.
\end{defn}

Note that for a semi-Anosov flow, the condition $r_u > 0$ is equivalent to ${d\alpha_u}_{\vert E^{wu}} > 0$. In other terms, $d\alpha_u$ is an \emph{exact} $2$-form positive on the leaves of $\mathcal{F}^{wu}$, implying that $\mathcal{F}^{wu}$ is taut. In the Anosov case, both $\mathcal{F}^{ws}$ and $\mathcal{F}^{wu}$ are taut.

\section{Contact and symplectic geometry}

In this section, we recall some basic notions from symplectic and contact geometry. We refer to~\cite{McDS} and~\cite{CE12} for extensive references.

        \subsection{Contact structures}

Let us first discuss contact structures with a focus on dimension three.

\paragraph{Basic features.}

A \textbf{contact structure} on a $3$-manifold $M$ is a `maximally non-integrable' plane field $\xi$. Concretely, the non-integrability can be (locally) phrased as the existence of a $1$-form $\alpha$ such that $\xi = \ker \alpha$ and $\alpha \wedge d\alpha \neq 0$. We will only consider \emph{cooriented} contact structures on \emph{oriented} $3$-manifold, so that such a \textbf{contact form} $\alpha$ is defined globally and evaluates positively on vectors positively transverse to $\xi$. Contact structures on $M$ come in two flavors: positive ($\alpha \wedge d\alpha >0$) and negative ($\alpha \wedge d\alpha < 0$) with respect of the orientation on $M$. Changing the coorientation does not change the sign of the contact structure in dimension three.

Two contact manifolds $(M, \xi)$ and $(M', \xi')$ are \textbf{contactomorphic} if there exists a diffeomorphism $\phi : M \rightarrow M'$ satisfying $\phi^*\xi' = \xi$.

By \textbf{Darboux theorem}, contact structures are locally standard: near every point $p \in M$, there exist local coordinates $(x,y,z)$ such that
$$\xi = \ker\big(dz \pm x dy\big).$$
The sign in front of $x dy$ determines the sign of $\xi$.

An object (vector field, curve) tangent to a contact structure $\xi$ is called \textbf{Legendrian}. If $X$ is a Legendrian vector field for $\xi$, then the Darboux coordinates from the last paragraph can be chosen such that $X$ becomes $\partial_x$. Then, the contact condition is equivalent to the (positive or negative) \emph{twisting} of $\xi$ about $X$.

A contact form for $\xi$ is not unique; any two such forms differ exactly by multiplication by a positive function. Given a contact form $\alpha$ for $\xi$, one defines its \textbf{Reeb vector field} $R_\alpha$ by
\begin{align*}
\alpha(R_\alpha) &= 1, \\
d\alpha(R_\alpha, \cdot \, ) &= 0.
\end{align*}
Note that the contact condition also means that $d\alpha$ is a nondegenerate $2$-form on $\xi$. From those equations, it is easy to see that the flow of $R_\alpha$ preserves $\xi$; reciprocally, any vector field $R$ transverse to $\xi$ and preserving $\xi$ is the Reeb vector field of some contact form for $\xi$ (namely, the unique contact form satisfying $\alpha(R)=1$).

An important feature of contact structures is that they are \emph{stable} objects, in the sense of \textbf{Gray stability}: if $(\xi_t)_{t \in [0,1]}$ is a path of contact structures, then there exists an isotopy $(\phi_t)_{t \in [0,1]}$ of $M$ such that $\phi_t^*\xi_t = \xi_0$ for every $t \in [0,1]$.

\paragraph{Tight vs.~overtwisted.} There exist two types of contact structures depending on their \emph{flexibility} or \emph{rigidity}. The standard overtwisted disk in $\R^3$ is the disk $D_\mathrm{OT} = \{r \leq \pi, \ z=0\}$ in cylindrical coordinates, with the contact structure $\xi_\mathrm{OT} = \ker \big( \cos(r)dz + r \sin(r) d\theta\big)$. Note that $\xi_\mathrm{OT}$ is tangent to $D_\mathrm{OT}$ along $\partial D_\mathrm{OT}$ and at the origin, and transverse to it elsewhere. An overtwisted disk in $(M, \xi)$ is a contact embedding of a neighborhood of $(D_\mathrm{OT}, \xi_\mathrm{OT})$ in $M$. A contact structure without overtwisted disks is \textbf{tight}. A fundamental result of Eliashberg~\cite{E89} is that overtwisted contact structures are completely classified in terms of their homotopy class as plane fields. On the other hand, tight contact structures are more geometric and are much harder to construct and classify.

A contact structure that admits a Reeb vector field without contractible closed orbits is called \textbf{hypertight} (not all of its Reeb vector fields need to satisfy this condition). Hofer~\cite{H93} showed that $3$-dimensional hypertight contact structures are tight.

        \subsection{Symplectic and Liouville structures}
We now discuss symplectic structures with a particular focus on $4$-dimensional manifolds.

\paragraph{Symplectic manifold.} A \textbf{symplectic form} on a manifold $V$ is a closed nondegenerate $2$-form $\omega$. If $\dim(V) =4$, the nondegeneracy of $\omega$ is equivalent to requiring that $\omega \wedge \omega$ is a volume form; in particular, $V$ is orientable and $\omega$ determines an orientation of $V$.

Darboux theorem for symplectic manifolds implies that symplectic forms are locally standard. On a $4$-manifold $V$, this means that every point $p \in V$ has a neighborhood with coordinates $(x_1, y_1, x_2, y_2)$ in which $\omega$ is of the form
$$\omega = dx_1 \wedge dy_1 + dx_2 \wedge dy_2.$$

Since a symplectic form $\omega$ is closed, it defines a cohomology class $[\omega] \in H^2(V; \R)$. By \textbf{Moser stability}, if $(\omega_t)_{t \in [0,1]}$ is a path of cohomologous symplectic forms on $V$ (supported on a compact subset away from $\partial V$), then there exists an isotopy $(\phi_t)_{t \in [0,1]}$ of $V$ such that $\phi_t^*\omega_t = \omega_0$ for every $t \in [0,1]$.  

\paragraph{Liouville structure.} We now focus on \emph{exact} symplectic manifolds, i.e., symplectic manifolds $(V, \omega)$ such that $\omega$ is exact. In that case, $V$ cannot be closed. If $V$ is compact with boundary $\partial V$, we say that a primitive $\lambda$ of $\omega$ is a \textbf{Liouville form} if its restriction on $\partial V$ is a contact form. We call the pair $(\omega, \lambda)$ a \textbf{Liouville structure} on $V$, and the triple $(V, \omega, \lambda)$ (or simply the pair $(V, \lambda)$) a \textbf{Liouville domain}.

If $(M, \xi)$ is a contact manifold, a \textbf{Liouville (or exact) filling} of $(M, \xi)$ is a compact Liouville manifold $(V, \lambda)$ such that $(\partial V, \ker \lambda_{\vert \partial V})$ is contactomorphism to $(M, \xi)$.

Since $\omega$ is nondegenerate, a primitive $\lambda$ of $\omega$ is equivalent to a vector field $Z$, defined by
$$\omega(Z, \cdot \ ) = \lambda.$$
This vector field $Z$ is called the \textbf{Liouville vector field} of $\lambda$. Alternatively, it is characterized by the relation
$$\mathcal{L}_Z \omega = \omega.$$
The condition that the restriction of $\lambda$ on $\partial V$ is contact is equivalent to $Z$ being positively transverse to $\partial V$. Writing $\alpha = \lambda_{\vert \partial V}$, there exists a tubular neighborhood of $\partial V$ in $V$ of the form $(-\epsilon, 0]_s \times \partial V$ in which $\lambda$ becomes $e^s \lambda$, and $Z$ becomes $\partial_s$. We can then attach a cylinder of the form $[0,+\infty) \times \partial V$ to $V$ along $\partial V$ and extend $\lambda$ by $e^s \alpha$ on $[0,+\infty) \times \partial V$. We call the resulting symplectic manifold $\big(\widehat{V}, \widehat{\omega}, \widehat{\lambda}\big)$ the \textbf{completion} of $(V, \omega, \lambda)$.

Two exact symplectic manifolds $(V, \omega, \lambda)$ and $(V', \omega', \lambda')$ are \textbf{exact symplectomorphic} is there exists a diffeomorphism $\phi : V \rightarrow V'$ satisfying that $\phi^*\lambda' = \lambda + df$ for some function $f : V \rightarrow \R$ with compact support disjoint from $\partial V$. We say that $(V, \omega, \lambda)$ is a \textbf{(finite type) Liouville manifold} is it is exact symplectomorphic to the completion of a Liouville domain. There is an obvious truncation procedure to obtain a Liouville domain from a Liouville manifold, such that its completion is exact symplectomorphic to the original Liouville manifold.

Combining Gray stability for contact structures and Moser stability for symplectic structures, one obtains the following stability result for Liouville domains. If $V$ is a compact manifold with boundary and $(\omega_t, \lambda_t)_{t \in [0,1]}$ is a path of Liouville structures on $V$, then there exists an isotopy $(\phi_t)_{t \in [0,1]}$ such that $\phi_t : (V, \lambda_0) \rightarrow (V, \lambda_t)$ is an exact symplectomorphism for all $t \in [0,1]$. In particular, $\phi_t$ restricts to a contactomorphism on $\partial V$ between the contact structures induced by $\lambda_0$ and $\lambda_t$. There is a similar statement for Liouville manifolds, under suitable taming hypothesis `at infinity'.

If $(V, \omega, \lambda)$ is a Liouville domain or manifold, with Liouville vector field $Z$, the $\textbf{skeleton}$ of $Z$ (or $\lambda$) is the compact set 
\begin{align} \label{eq:skel}
\mathfrak{skel}(V, \lambda) \coloneqq \bigcup_{\substack{K \subset V \\\mathrm{compact}}} \bigcap_{t > 0} \phi^{-t}_Z(K)
\end{align}
where $\phi_Z^t$ denotes the time-$t$ flow of $Z$. In other words, $\mathfrak{skel}(V, \lambda)$ is the (compact) subset of $V$ onto which $V$ retracts under the backward flow of $Z$. It heavily depends on the choice of the Liouville structure $\lambda$ (and not only on $\omega$).

In general, the skeleton of a Liouville domain can be extremely complicated, since the dynamics of $Z$ can be quite chaotic. There is a particular class of Liouville domains, called \textbf{Weinstein domains}, for which the Liouville dynamics is relatively tame. A Liouville domain is Weinstein if there exists a Morse function on $V$ for which $Z$ is gradient-like. Alternatively, Weinstein domains can be defined by attaching specific handles called \emph{Weinstein handles}. The skeleton is exactly the closure of the union of the stable manifolds of the handles. It turns out that these handles have index at most half the dimension of $V$; in dimension $4$, the index is at most $2$ implying that $\partial V$ is connected and $H_3(V; \Z)$ vanishes. We define a \textbf{(finite type) Weinstein manifold} as a Liouville manifold which is exact symplectomorphic to the completion of a Weinstein domain.

Not every Liouville domain is Weinstein. The first example of a Liouville manifold which doesn't admit Weinstein structures was constructed by McDuff~\cite{McD}. Geiges constructed more examples~\cite{Gei}, and Mitsumatsu established a general construction from $3$-dimensional Anosov flows in~\cite{M95}. Mitsumatsu's construction will be reviewed in Section~\ref{sec:Mit}.

\paragraph{Liouville pairs.} We now consider special Liouville structures on $4$-manifolds of the form $[-1, 1] \times M$ or $\R \times M$, where $M$ is a closed, oriented and connected $3$-manifold.

A pair of contact forms $(\alpha_-, \alpha_+)$ on $M$ is an \textbf{exponential Liouville pair} if the $1$-form $\lambda$ on $V =\R_s \times M$ defined by 
\begin{align} \label{eq:expliouv}
    \lambda = \lambda_\mathrm{exp} \coloneqq e^{-s} \alpha_- + e^s \alpha_+
\end{align}
is Liouville. Because of the specific form of $\lambda$, it is a Liouville form if and only if $\alpha_-$ is a negative contact form, $\alpha_+$ is a positive contact form, and $d\lambda$ is a nondegenerate $2$-form. In that case, $(V, \lambda)$ is a \emph{non-Weinstein} Liouville manifold. For $A > 0$ large enough, $[-A, A] \times M \subset V$ is a Liouville domain whose completion is exact symplectomorphic to $(V, \lambda)$.

An alternative definition exists in the literature. A pair of contact forms $(\alpha_-, \alpha_+)$ on $M$ is an \textbf{linear Liouville pair} if the $1$-form $\lambda$ on $V =[-1,1]_t \times M$ defined by
\begin{align} \label{eq:linliouv}
   \lambda = \lambda_\mathrm{lin} \coloneqq (t-1) \alpha_- + (t+1) \alpha_+ 
\end{align}
is Liouville. In that case, $(V, \lambda)$ is a non-Weinstein Liouville \emph{domain}.

We warn the reader that these two notions are \textbf{not} equivalent: there exists pairs of contact forms which are exponential Liouville pairs but are not linear Liouville pairs. However, a pair of contact forms which is both a linear and an exponential Liouville pair induce the same Liouville structures. We refer to~\cite[Section 5.2]{Mas23} for a more thorough discussion.

Not every $3$-manifold admits a Liouville pair. For instance, we showed in~\cite{Mas24} that the existence of a (linear or exponential) Liouville pair implies the existence of a taut foliation without closed leaves on $M$. However, we will see below that Anosov flows constitute an important source of Liouville pairs.

\section{The Mitsumatsu construction} \label{sec:Mit}

We now present a construction due to Mitsumatsu~\cite{M95} and independently by Eliashberg--Thurston~\cite{ET}, and later extended and streamlined by Hozoori~\cite{H23}.

        \subsection{Projectively Anosov case}

Let $\Phi$ be a projectively Anosov flow on $M$ generated by the smooth vector field $X$, and $(\alpha_s, \alpha_u)$ be a defining pair as in Definition~\ref{def:defpair}. We define:
\begin{align}
    \alpha_- &\coloneqq \alpha_u + \alpha_s, \label{eq:alpha-}\\
    \alpha_+ &\coloneqq \alpha_u - \alpha_s. \label{eq:alpha+}
\end{align}
Those $1$-forms are continuous and continuously differentiable along $X$. Then, by the definition of a defining pair, it is easy to compute:
\begin{align}
    \alpha_- \wedge \mathcal{L}_X \alpha_- &= (r_u - r_s) \, \alpha_s \wedge \alpha_u,\\
    \alpha_+ \wedge \mathcal{L}_X \alpha_+ &= -(r_u - r_s) \, \alpha_s \wedge \alpha_u.
\end{align}
If $\alpha_-$ and $\alpha_+$ were $\mathcal{C}^1$, this would imply that they are negative and positive contact forms, respectively, since $r_u - r_s > 0$. While they are not assumed to be $\mathcal{C}^1$, they can be approximated by smooth $1$-forms, denoted $\widetilde{\alpha}_-$ and $\widetilde{\alpha}_+$, still satisfying $\widetilde{\alpha}_-(X) = \widetilde{\alpha}_+(X)=0$, and which are contact; see~\cite{Hoz24, Mas23} for more details. As a result, we obtain two contact structures with opposite signs $\xi_-$ and $\xi_+$ satisfying $\xi_- \pitchfork \xi_+ = \langle X \rangle$. Such a pair is called a \textbf{bicontact structure supporting $\Phi$}.

\begin{defn}
    A \textbf{bicontact structure} on $M$ is a pair $(\xi_-, \xi_+)$ of transverse negative and positive contact structures.
\end{defn}

Reciprocally, Mitsumatsu and Eliashberg--Thurston showed that any bicontact structure $(\xi_-, \xi_+)$ on $M$ supports an projectively Anosov flow, in the sense that any nonsingular $\mathcal{C}^1$ vector field $X \in \xi_- \cap \xi_+$ generates a projectively Anosov flow. In conclusion:

\begin{thm} \label{thm:projanosov}
    A nonsingular flow $\Phi$ generated by a smooth vector field $X$ is projectively Anosov if and only if it is supported by a bicontact structure.
\end{thm}

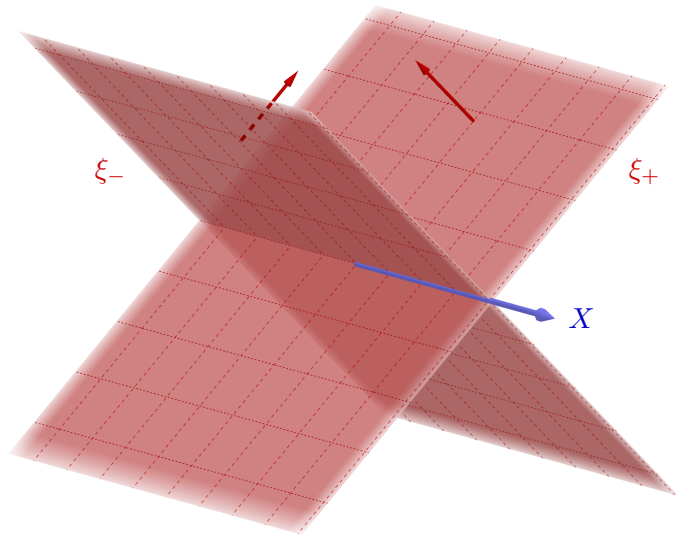
\begin{figure}[h!]
    \begin{center}
        \begin{picture}(90, 75)(0,0)
        \put(0,0){\includegraphics[width=90mm]{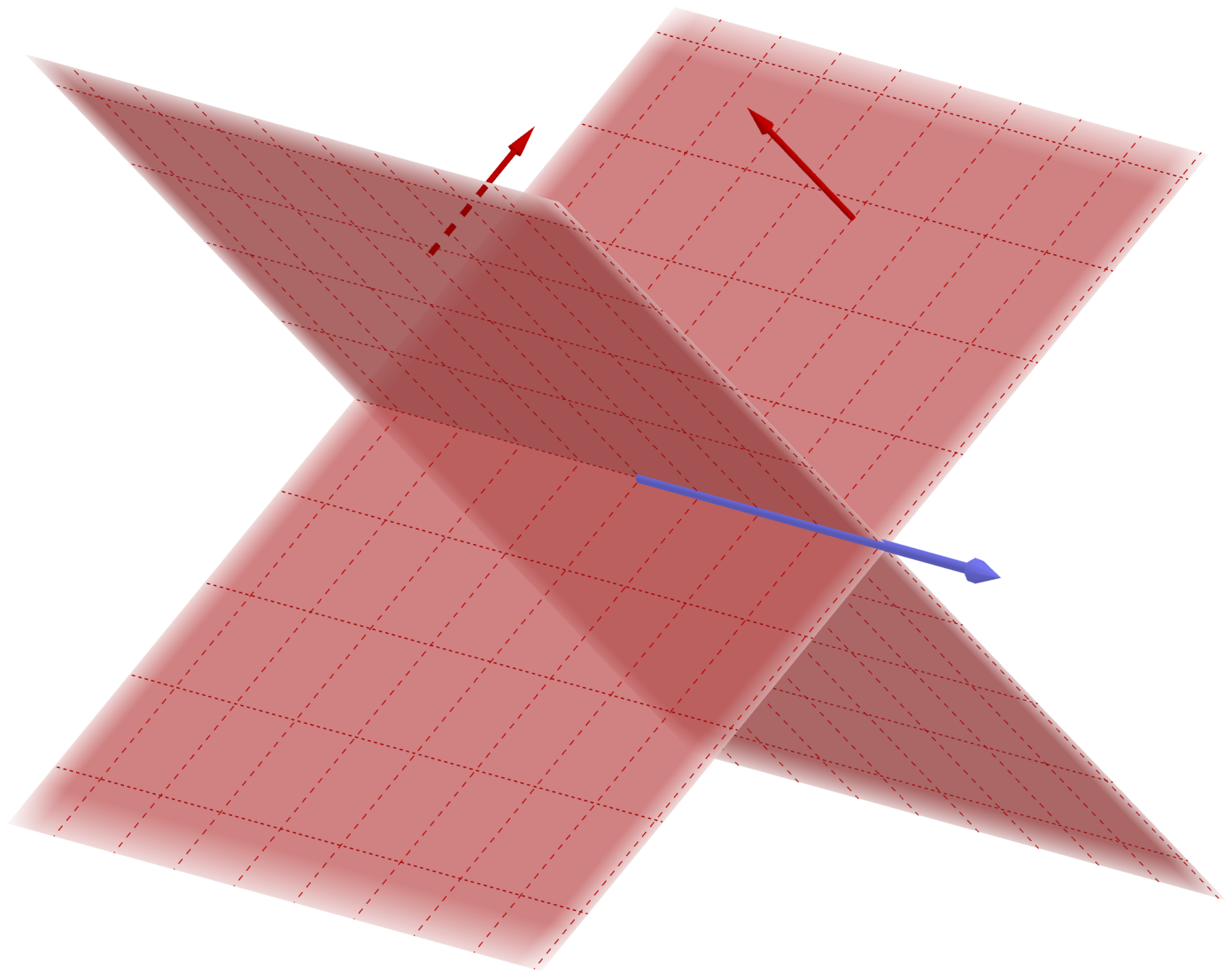}}
        \color{red1}
        \put(12,48){$\xi_-$}
        \put(83,48){$\xi_+$}
        \color{blue1}
        \put(75,28){$X$}
        \end{picture}
        \captionsetup{width=100mm}
        \caption{Bicontact structures supporting a vector field/flow.}
    \label{fig:bicontact}
    \end{center}
\end{figure}

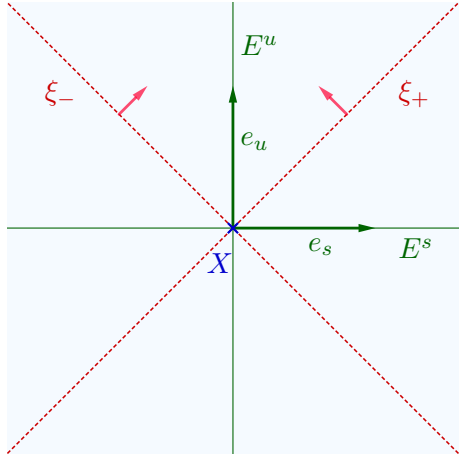
\begin{figure}[h!]
    \begin{center}
        \begin{picture}(60, 60)(0,0)
        \put(0,0){\includegraphics[width=60mm]{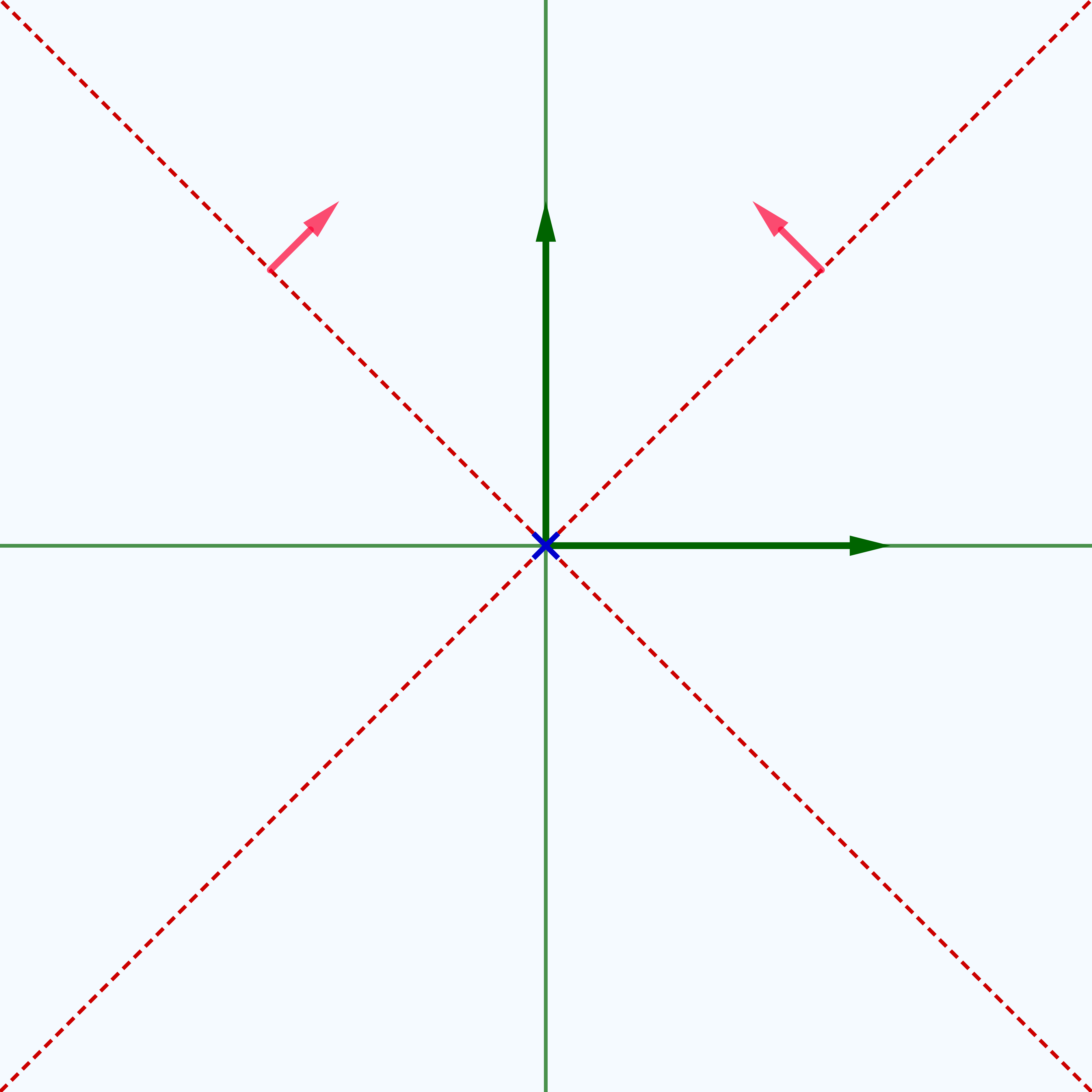}}
        \color{blue1}
        \put(26.5,24){$X$}
        \color{green1}
        \put(40,27){$e_s$}
        \put(52,26){$E^s$}
        \put(31,53){$E^u$}
        \put(31,41){$e_u$}
        \color{red1}
        \put(5,47){$\xi_-$}
        \put(52,47){$\xi_+$}
        \end{picture}
        \captionsetup{width=90mm}
        \caption{Anosov/dominated splitting and supporting bicontact structure.}
        \label{fig:Anosovorientation}
    \end{center}
\end{figure}

        \subsection{Semi-Anosov case}

Let us now assume that $\Phi$ is semi-Anosov and $(\alpha_s, \alpha_u)$ is a defining pair for $\Phi$. We define $\alpha_\pm$ as in equations~\eqref{eq:alpha-} and~\eqref{eq:alpha+}, and we define a $1$-form $\lambda$ on $[-1,1]_t \times M$ as in~\eqref{eq:linliouv}
\begin{align*}
    \lambda &\coloneqq (1-t) \alpha_- + (1+t) \alpha_+\\
        &= 2(\alpha_u - t \alpha_s). \nonumber
\end{align*}
As before, since $\alpha_s$ and $\alpha_u$ are not $\mathcal{C}^1$ in general, $\lambda$ is not $\mathcal{C}^1$. For simplicity, let us pretend that all these forms are $\mathcal{C}^1$ (the general case can be treated using the approximation methods of~\cite{Hoz24, Mas23}). Then a straightforward computation shows:
\begin{align*}
    \iota_X \iota_{\partial_t} (d\lambda \wedge d\lambda) &= 2\iota_X(\mathcal{L}_{\partial_t} \lambda \wedge d\lambda)\\
    &= 4 \iota_X(-\alpha_s \wedge d_\lambda) \\
    &= 4 \alpha_s \wedge \mathcal{L}_X \lambda \\
    &= 8 \alpha_s \wedge (r_u \alpha_u - t r_s \alpha_s) \\
    &= 8r_u \alpha_s \wedge \alpha_u,
\end{align*}
hence $\omega = d\lambda$ is symplectic since $r_u > 0$. Moreover, along the boundary components of $[-1,1] \times M$, $\lambda$ is a contact form for $\xi_-$ and $\xi_+$, respectively, hence $\lambda$ is a \emph{Liouville filling} of $(M, \xi_+) \sqcup (-M, \xi_-)$. With the terminology of the previous section, $(\alpha_-, \alpha_+)$ is a linear Liouville pair on $[-1,1] \times M$. We will call this construction the \textbf{standard construction} of a linear Liouville pair supporting a semi-Anosov flow.

As we already mentioned, there is a slightly different notion of Liouville pair in the literature which is also relevant. Instead of a linear interpolation as in~\eqref{eq:linliouv}, one can also consider an exponential interpolation on the noncompact manifold $\R_s \times M$ as in~\eqref{eq:expliouv}.

Similar computations as before show that the standard construction also yields an exponential Liouville pair (see~\cite{Mas23} for a detailed proof). We warn the reader that although similar, the notions of linear and exponential Liouville pairs are \underline{different}, see~\cite[Lemma 5.6]{Mas23}. However, the linear and exponential Liouville structures defined by the standard construction are equivalent, see~\cite[Section 5.6]{Mas23}.

\begin{defn}
    A \textbf{transverse (linear or exponential) Liouville pair} is a (linear or exponential) Liouville pair $(\alpha_-, \alpha_+)$ such that $\xi_- = \ker \alpha_-$ and $\xi_+ = \ker \alpha_+$ are transverse contact structures.
\end{defn}

It follows that a transverse Liouville pair $(\alpha_-, \alpha_+)$ defines a bicontact structure $(\xi_-, \xi_+)$, and therefore the intersection $\xi_- \cap \xi_+$ is generated by a projectively Anosov flow. Moreover, the Liouville condition implies that this flow (suitably oriented as in~\cite[Definition 2.1]{Mas23}, see also Figure~\ref{fig:bicontact}) is semi-Anosov; see~\cite{Hoz24-reg, Mas24}. Therefore, we get:

\begin{thm}
    A nonsingular flow $\Phi$ generated by a smooth vector field $X$ is semi-Anosov if and only if it is supported by a transverse Liouville pair.
\end{thm}

In view of Lemma~\ref{lem:semianosov}, we immediately obtain:

\begin{thm} \label{thm:2ALpairs}
    A nonsingular flow $\Phi$ generated by a smooth vector field $X$ is Anosov if and only if both $\Phi$ and $\Phi^{-1}$ are supported by transverse Liouville pairs.
\end{thm}

This result characterizes Anosov flows purely in terms of contact and symplectic geometry. We will encounter many more such characterizations below.

        \subsection{Reeb vector fields}

A special feature of the standard construction of contact forms from a semi-Anosov flow is that the induced Reeb vector fields satisfy specific transversality properties with respect to the (semi-)Anosov foliations. More specifically, if $(\alpha_s, \alpha_u)$ is a defining pair supporting the semi-Anosov flow $\Phi$, then the Reeb vector fields $R_\pm$ of the contact forms $\alpha_\pm$ obtained from the standard construction (with appropriate smoothing) satisfy
\begin{align*}
    \alpha_s(R_-) &> 0, \\
    \alpha_s(R_+) &< 0,
\end{align*}
hence $R_-$ (resp.~$R_+$) is positively (resp.~negatively) transverse to $\mathcal{F}^{wu}$. Since $\mathcal{F}^{wu}$ is taut, $R_\pm$ cannot have contractible closed orbits, and the contact structures $\xi_\pm$ are \emph{hypertight}. If moreover $\Phi$ is Anosov and $(\alpha_s, \alpha_u)$ is a defining pair for it, then we further have
\begin{align*}
    \alpha_u(R_-) &> 0, \\
    \alpha_u(R_+) &> 0,
\end{align*}
and $R_\pm$ are both positively transverse to $\mathcal{F}^{ws}$. The details of these computations can be found in~\cite[Section 3.2]{Mas23}. Following~\cite{Hoz24}, we say that a vector field $R$ is \textbf{dynamically negative} if it is positively transverse to $E^{ws}$ and negatively transverse to $E^{wu}$. The existence of dynamically negative/positive Reeb flows for the standard construction are crucially used in~\cite{CLMM}.

Anosov and semi-Anosov flows can also be characterized by the transversality properties of certain Reeb vector fields for bicontact structures supporting them. More precisely, Hozoori showed in~\cite{Hoz24, Hoz24-adapt}:

\begin{thm} \label{thm:reebanosov}
    Let $\Phi$ be a projectively Anosov flow. 
    \begin{enumerate}
        \item $\Phi$ is Anosov if and only if it is supported by a bicontact structure $(\xi_-, \xi_+)$ such that $\xi_+$ admits a dynamically negative Reeb vector field.
        \item Moreover, if $\Phi$ is Anosov then there exists a supporting bicontact structure $(\xi_-, \xi_+)$ such that $\xi_-$ contains a Reeb vector field for $\xi_+$.
    \end{enumerate}
\end{thm}

The proof of~\cite[Theorem 6.3]{Hoz24} can also be adapted to the semi-Anosov case to show:

\begin{thm}
    Let $\Phi$ be a projectively Anosov flow. Then $\Phi$ is semi-Anosov if and only if it is supported by a bicontact structure $(\xi_-, \xi_+)$ such that $\xi_+$ admits a Reeb vector field which is negatively transverse to $E^{wu}$.
\end{thm}

Similar statements hold for Reeb vector fields of $\xi_-$.

\section{Anosov Liouville structures}

In the previous section, we detailed a construction of Liouville structures from (semi-)Anosov flows. However, this construction depends on various choices (defining pairs, appropriate smoothings). We now introduce a more abstract notion of Liouville structures that generalizes the standard construction of Mitsumatsu. 

        \subsection{Definition and contractibility}

The following definition already appeared in~\cite{Mas23, CLMM}. 

\begin{defn}
    An \textbf{Anosov Liouville structure} (or AL structure) on $\R \times M$ supporting an Anosov flow $\Phi$ is a Liouville structure $(\omega, \lambda)$ such that $\lambda$ is of the form of~\eqref{eq:expliouv} for a transverse Liouville pair $(\alpha_-, \alpha_+)$ supporting $\Phi$.
\end{defn}

This notion should be called an \emph{exponential} Anosov Liouville structure; there is a similar definition of \emph{linear} Anosov Liouville structure on the Liouville domain $[-1,1] \times M$. As noted before, the standard construction yields exponential and linear Liouville structures which are equivalent (exact symplectomorphic after completion). However, we only state the following theorem for exponential AL structures, as we do not know if the equivalent statement holds for linear AL structures.

\begin{thm}[\cite{Mas23}] \label{thm:ALcontr}
    The space of (exponential) AL structures supporting a given Anosov flow $\Phi$ on $M$ is contractible.
\end{thm}

As a result, the standard construction of Mitsumatsu yields a Liouville structure which is canonical and independent of all choices up to homotopy. In Section~\ref{sec:gAL} below, we propose a generalization of the notion of (exponential or linear) Liouville pair which is more natural from the perspective of symplectic geometry and we show that those `generalized AL structures' form a contractible space. As a result, all linear AL structures supporting a given flow are also equivalent, through a Liouville homotopy that might not remain within linear AL structures. Related results on the equivalence between linear/exponential and more general Liouville structures can be found in~\cite{Hoz24-reg}. The same results hold for Liouville structures supporting semi-Anosov flows.

\medskip

An important consequence of Theorem~\ref{thm:ALcontr} is that the symplectic invariants of an AL structure supporting an Anosov flow $\Phi$ are invariants of $\Phi$ itself. This includes all the Floer-theoretic invariants, e.g., symplectic cohomology and wrapped Fukaya category. These invariants are thoroughly studied in~\cite{CLMM}. The methods involved in the proof of Theorem~\ref{thm:ALcontr} and related results actually show that those AL structures are invariants of \emph{homotopy classes} of Anosov flows, in the sense that a path of Anosov flows `lifts' to a path of AL structures (see~\cite{Mas23} for more details). It is natural to ask if a similar invariance holds for two \emph{orbit equivalent} Anosov flows:

\begin{question} \label{quest:orbit1}
    If $\Phi_0$, $\Phi_1$ are two Anosov flows on $M$ with associated (linear, exponential, generalized) AL structures $(\omega_0, \lambda_0)$ and  $(\omega_1, \lambda_1)$. If $\Phi_0$ and $\Phi_1$ are orbit equivalent, are $(\omega_0, \lambda_0)$ and  $(\omega_1, \lambda_1)$ exact symplectomorphic?
\end{question}

In work in progress with Jonathan Bowden, we give an affirmative answer to this question, see the next subsection  for more details. In particular, all the aforementioned invariants are \emph{topological invariants} for Anosov and could be relevant to the classification of Anosov flows. It would be interesting to investigate if Anosov flows can be \emph{characterized} by their induced AL structures:

\begin{question} \label{quest:orbit2}
    With the notations of Question~\ref{quest:orbit1}, if $(\omega_0, \lambda_0)$ and  $(\omega_1, \lambda_1)$ exact symplectomorphic, are $\Phi_0$ and $\Phi_1$ orbit equivalent ?
\end{question}

By the work of Barthelm\'{e}--Mann--Fenley~\cite{BFM23}, transitive Anosov flows are essentially characterized by the free homotopy classes of their closed orbits (at least on an atoroidal manifold). This motivates the following

\begin{question} \label{quest:orbit5}
    Is it possible to recover the set of free homotopy classes of closed orbits of an Anosov flow from the Floer-theoretic invariants of its associated AL structure?
\end{question}

Notice that the closed orbits of the Anosov flow are Legendrian for both contact structures $\xi_\pm$ coming from a supporting bicontact structure. Moreover, if $\gamma$ is such a closed orbit, then $L_\gamma \coloneqq [-1,1] \times \gamma \subset [-1,1] \times M$ is \emph{Lagrangian} for any AL structure $(\omega, \lambda)$ supporting the flow, in the sense that $\omega_{\vert L_\gamma} \equiv 0$. We actually have $\lambda_{\vert L_\gamma} \equiv 0$, so that those are (strictly) exact Lagrangians, and constitute objects of the wrapped Fukaya category associated with the flow, see~\cite{CLMM}. Perhaps one could upgrade this category to include the information of the free homotopy classes of the (exact) Lagrangian submanifolds. This could potentially provide a path to an affirmative answer to Question~\ref{quest:orbit5}.

            \subsection{Effect of orbit equivalence} \label{sec:orbit}

In this section, we consider two Anosov flows $\Phi_0$ and $\Phi_1$ on $M$, and we assume that their Anosov splittings are orientable. We further assume that $\Phi_0$ and $\Phi_1$ are orbit equivalent, via an orbit equivalence $h : M \rightarrow M$. We do not assume that $h$ is orientation-preserving. It is well known that $h$ sends the weak stable and unstable foliations of $\Phi_0$ to the ones of $\Phi_1$, but this does not hold for the strong ones. We will assume that the weak foliations of $\Phi_0$ and $\Phi_1$ are oriented and that $h$ preserves their orientations. We further consider AL structures $(\omega_0, \lambda_0)$ and $(\omega_1, \lambda_1)$ on $V = [-1,1] \times M$ supporting $\Phi_0$ and $\Phi_1$, respectively, for the chosen orientations (we might have to consider two different orientations on $M$ for the two flows). The following result which answers Question~\ref{quest:orbit1} positively will appear in~\cite{BM}:

\begin{thm} \label{thm:orbiteq}
    With the previous notations and assumptions, $(\omega_0, \lambda_0)$ and $(\omega_1, \lambda_1)$ are exact symplectomorphic, via an exact symplectomorphism isotopic to $\mathrm{id} \times h$.
\end{thm}

This will be a consequence of a more general result about Liouville structures induced by suitable taut foliations, applied to the weak-unstable foliations of $\Phi_0$ and $\Phi_1$. We immediately obtain:

\begin{cor} \label{cor:bicont}
    If $(\xi^0_-, \xi^0_+)$ and $(\xi^1_-, \xi^1_+)$ are bicontact structures supporting $\Phi_0$ and $\Phi_1$, respectively, then $\xi^0_-$ and $\xi^1_-$ (resp.~$\xi^0_+$ and $\xi^1_+$) are contactomorphic (by possibly different contactomorphisms).
\end{cor}

We will also show a more precise result:

\begin{thm}[\cite{BM}]
    With the notations of Corollary~\ref{cor:bicont}, $(\xi^0_-, \xi^0_+)$ and $(\xi^1_-, \xi^1_+)$ are \emph{deformation equivalent} through bicontact structures. More precisely, there exists a diffeomorphism $\widetilde{h} : M \rightarrow M$, isotopic to $h$, such that $(\widetilde{h}^*\xi^1_-, \widetilde{h}^*\xi^1_+)$ and $(\xi^0_-, \xi^0_+)$ are homotopic through bicontact structures.
\end{thm}

In view of the correspondence between bicontact structures and projectively Anosov flows from Theorem~\ref{thm:projanosov}, we obtain:

\begin{cor}
    There exists a diffeomorphism $\widetilde{h} : M \rightarrow M$ isotopic to $h$ such that $\Phi_0$ and $\widetilde{h}^*\Phi_1$ are homotopic through projectively Anosov flows.
\end{cor}

In that case, we will say that $\Phi_0$ and $\Phi_1$ are \textbf{projectively deformation equivalent}. A converse to this result was conjectured by Hozoori~\cite[Conjecture 7.3]{Hoz24}:

\begin{conj}
    If two (oriented) Anosov flows on $M$ are projectively deformation equivalent, then they are orbit equivalent.
\end{conj}

If true, this would provide a purely contact geometric characterization of orbit equivalence between Anosov flows, as a deformation equivalence between their associated bicontact structures. We might further ask:

\begin{question} \label{quest:orbit4}
    Let $\Phi_0$ and $\Phi_1$ be two (oriented) Anosov flows supporting bicontact structures $(\xi^0_-, \xi^0_+)$ and $(\xi^1_-, \xi^1_+)$, respectively. If $\xi^0_-$ is contact homotopic to $\xi^1_-$ and $\xi^0_+$ is contact homotopic to $\xi^1_+$, are $\Phi_0$ and $\Phi_1$ orbit equivalent?
\end{question}

            \subsection{Anosov Liouville pairs}

We saw in Theorem~\ref{thm:2ALpairs} that Anosov flows can be characterized in terms of the existence of two transverse Liouville pairs supporting the flow and its inverse. It is natural to ask whether one can find a single Liouville pairs adapted to both the flow and its inverse. It turns out that if $\Phi$ is an Anosov flow, then the standard construction from Section~\ref{sec:Mit} induces contact forms $\alpha_\pm$ such that both $(\alpha_-, \alpha_+)$ and $(-\alpha_-, \alpha_+)$ are (linear and exponential) Liouville pairs supporting $\Phi$ and its inverse, respectively. This motivates the following

\begin{defn}
    An \textbf{exponential (resp.~linear) Anosov Liouville pair} (or AL pair for short) is a pair of contact forms $(\alpha_-, \alpha_+)$ such that both $(\alpha_-, \alpha_+)$ and $(-\alpha_-, \alpha_+)$ are exponential (resp.~linear) Liouville pairs.
\end{defn}

Notice that this definition is purely contact and symplectic geometric and does not make any reference to Anosov flows. It turns out that an AL pair automatically induces contact structures which are transverse and which intersect along a projectively Anosov flow. Then by Theorem~\ref{thm:2ALpairs}, this flow is automatically Anosov. The latter is only defined up to positive time reparametrization. If $\mathcal{AL}$ denotes the space of exponential AL pairs on $M$, and $\mathcal{AF}$ denotes the space of Anosov flows up to positive time reparametrization (or equivalently, the space of unit Anosov vector fields for a fixed underlying metric), we get a well-defined map
$$\mathcal{I}:\mathcal{AL} \longrightarrow \mathcal{AF}$$
sending an AL pair $(\alpha_-, \alpha_+)$ to a flow generating $\ker \alpha_- \cap \ker \alpha_+$ with appropriate orientation (see Figure~\ref{fig:bicontact}). A similar map can be defined on the space of \emph{linear} AL pairs; however, we do not know if the following result holds for it.

\begin{thm}[\cite{Mas23}]
    The map $\mathcal{I}$ is a Serre fibration with contractible fibers, hence is a homotopy equivalence.
\end{thm}

It would be interesting to use this result to study the homotopy type of the space of Anosov flows on a given manifold. It is known that this space may have infinitely many connected component and nontrivial fundamental group, see~\cite{M13}. Further properties of (linear and exponential) AL pairs are studied in~\cite{Mas23}.

        \subsection{Volume preserving Anosov flows}

We say that a flow $\Phi$ on $M$ generated by a smooth vector field $X$ is \textbf{volume preserving} if there exists a smooth volume form $\mathrm{dvol}$ satisfying
$$\mathcal{L}_X \mathrm{dvol} = 0.$$
In other words, $X$ is divergence-free with respect to $\mathrm{dvol}$. Hozoori showed in~\cite{H23} (see also~\cite[Lemma 3.4]{Mas23}) that volume preserving Anosov flows can be characterized in terms of their expansion rates $r_s$ and $r_u$:

\begin{thm}[\cite{H23}]
    A projectively Anosov flow $\Phi$ is volume preserving Anosov if and only if it admits a defining pair $(\alpha_s, \alpha_u)$ with expansion rates $(r_s, r_u)$ satisfying $r_s + r_u = 0$. 
\end{thm}

Alternatively, volume preserving Anosov flows can be characterized geometrically in terms of the Reeb vector fields of supporting bicontact structures:

\begin{thm}[\cite{H23}]
    A projectively Anosov flow $\Phi$ is volume preserving Anosov if and only if it admits a supporting ($\mathcal{C}^1$) bicontact structure $(\xi_-, \xi_+)$ with Reeb vector fields $R_-$ and $R_+$ satisfying $R_\pm \in \xi_\mp$.
\end{thm}

Compare with the second item of Theorem~\ref{thm:reebanosov}.

\begin{rem}
    We say that a flow $\Phi$ is \textbf{(topologically) transitive} if it has a dense orbit. It is well-known (for instance using \emph{Smale's spectral decomposition}) that volume preserving Anosov flows are transitive. Conversely, Asaoka~\cite{A08} showed that (codimension-one) transitive Anosov flows are orbit equivalent to volume preserving ones. The first example of nontransitive Anosov flow in dimension three was constructed by Franks--Williams~\cite{FW80}.
\end{rem}

In~\cite[Theorem 1.4]{Hoz24}, Hozoori shows that every bicontact structure supporting an Anosov flow is the kernel of a \emph{linear} Liouville pair, and the same proof holds for semi-Anosov flows:

\begin{thm}[\cite{Hoz24}]
    If $\Phi$ is a semi-Anosov flow on $M$, then every bicontact structure supporting $\Phi$ is the kernel of a linear Liouville pair.
\end{thm}

We also showed a similar result for exponential Liouville pairs in the context of \emph{volume preserving} Anosov flows in~\cite[Section 3.4]{Mas23}:

\begin{prop}[\cite{Mas23}]
    If $\Phi$ is a volume preserving Anosov flow, then every bicontact structure supporting $\Phi$ is the kernel of an exponential (Anosov) Liouville pair.
\end{prop}

However, we showed in~\cite[Theorem 3.15]{Mas23} that the situation is more complicated for exponential Liouville pairs in general: the only Anosov flows such that every supporting bicontact structure is the kernel of an \emph{exponential Anosov Liouville pair} are the volume preserving one. We generalize this result further by considering exponential Liouville pairs instead of exponential \emph{Anosov} Liouville pairs:

\begin{thm} \label{thm:volumepres}
    Let $\Phi$ be an Anosov flow on $M$, and assume that every bicontact structure supporting $\Phi$ is the kernel of an exponential Liouville pair. Then $\Phi$ is volume preserving.
\end{thm}

\begin{proof}
    We follow the strategy of the proof of~\cite[Theorem~3.15]{Mas23}, but we will use different arguments than in~\cite[Proposition~B.1]{Mas23}. In fact, we will give a more direct proof using the main result of Lopes--Thieullen~\cite{LT05}.

    Computations as in the proof of~\cite[Theorem 3.15]{Mas23} show that in the present setting, for any defining pair $(\alpha_s, \alpha_u)$ for $\Phi$ with corresponding expansion rates $r_s$ and $r_u$ satisfying $r_s < 0 < r_u$, the following holds:

    \begin{center}
    \begin{flushleft}
        \textit{\textbf{Claim.} For every $\epsilon > 0$, there exists a smooth function $h_\epsilon : M \rightarrow \R$ such that 
        \begin{align*}
            X \cdot h_\epsilon + r_s + r_u \geq -\epsilon. \label{eq:rurs}
        \end{align*}}
    \end{flushleft}
    \end{center}
    In particular, this readily implies that for every periodic orbit $(\phi^t(x))_{0 \leq t \leq T}$ of $\Phi$ with period $T$,
    \begin{align*}
        \int_0^T (r_s + r_u)\circ \phi^t(x) \, dt \geq 0.
    \end{align*}
    We may further assume that $r_s$ and $r_u$ are $\mathcal{C}^1$ (see~\cite[Lemma 3.5]{Mas23}), so that the main result of~\cite{LT05} applies and there exists a (H\"{o}lder) continuous function $h : M \rightarrow \R$, continuously differentiable along $X$, satisfying
    \begin{align}
        \delta \coloneqq X \cdot h + r_s + r_u \geq 0.
    \end{align}
    It is easy to check that the quantity $r_s + r_u$ corresponds to the divergence of $X$ for the $\mathcal{C}^1$ volume form 
    $$\mathrm{dvol}\coloneqq \alpha_s \wedge \alpha_u \wedge \theta,$$
    where $\theta$ is any smooth $1$-form satisfying $\theta(X)\equiv 1$. We then consider the volume form $\mathrm{dvol}' \coloneqq e^h \mathrm{dvol}$ which satisfies 
    $$\mathcal{L}_X \mathrm{dvol}' = \delta \, \mathrm{dvol}',$$
    so that the divergence $\delta$ of $X$ for $\mathrm{dvol}'$ is nonnegative. In particular, $(\phi^t)^* \mathrm{dvol}'\geq \mathrm{dvol}'$ for every $t \geq 0$. This implies that $X$ \emph{preserves} $\mathrm{dvol}'$:
    \begin{align*}
        0 \leq \int_M \delta \, \mathrm{dvol}' &= \int_M \mathcal{L}_X \mathrm{dvol}' \\
        &= \int_M \lim_{t \rightarrow 0} \frac{1}{t}\big((\phi^t)^*\mathrm{dvol}' - \mathrm{dvol}'\big) \\
        &\leq \liminf_{t \rightarrow 0} \frac{1}{t} \left( \int_M (\phi^t)^*\mathrm{dvol}' - \int_M \mathrm{dvol}' \right) \\
        &= 0
    \end{align*}
    by Fatou's lemma, hence $\delta \equiv 0$ since $\delta \geq 0$ is continuous. By~\cite[Corollary 2.1]{LMM}, we conclude that $\mathrm{dvol}'$ is a \emph{smooth} invariant volume form for $\Phi$.
\end{proof}

Combining the two previous theorems, we readily get:

\begin{cor}
    If $\Phi$ is an Anosov flow which is \underline{not} volume preserving (e.g., $\Phi$ is not transitive), then there exists a linear Liouville pair supporting $\Phi$ which is not an exponential Liouville pair, and whose underlying bicontact structure is not the kernel of an exponential Liouville pair.
\end{cor}

\section{Liouville dynamics}

In this section, we describe some features of the dynamics of the Liouville vector fields of Anosov Liouville structures, with a particular focus on their \emph{skeleton}; see~\eqref{eq:skel}. The content of this section is mostly based on~\cite[Section 4.4]{Mas24} and~\cite{Hoz24-reg}. We refer to the latter article for an extensive study of these topics.

            \subsection{The skeleton}

In general, the skeleton of a Liouville structure can be extremely complicated. However, we showed that the skeleton of a Liouville structure defined from an exponential Liouville pair (not necessarily transverse) is always a codimension-$1$ submanifold of $\R \times M$ homeomorphic to $M$. A similar result holds for linear Liouville pairs. More precisely:

\begin{thm}
    Let $(\alpha_-, \alpha_+)$ be a Liouville pair on $M$ and $\lambda$ be the Liouville form defined by~\eqref{eq:expliouv}. Then there exists a continuous map $\sigma : M \rightarrow \R$ such that
    $$\mathfrak{skel}(\lambda) = \mathrm{graph}(\sigma) \coloneqq \{ (\sigma(p), p) \ \vert \ p \in M \}.$$
\end{thm}

If $(\alpha_-, \alpha_+)$ is a transverse Liouville pair supporting a semi-Anosov flow $\Phi$, then the map $\sigma$ is characterized by
$$\ker\left\{e^{-\sigma} \alpha_- + e^\sigma \alpha_+\right\} = E^{ws}.$$
Then, the regularity of $\sigma$ is related to the regularity of the weak-stable bundle of $\Phi$, which is not necessarily $\mathcal{C}^1$. However, it is $\mathcal{C}^1$ when $\Phi$ is Anosov by a result of Hasselblatt. We refer to~\cite{Hoz24-reg} for more results about the regularity of the skeleton.

To summarize, the skeleton of an AL structure is a $\mathcal{C}^1$ codimension-$1$ submanifold of $\R \times M$ which is graphical over $\{0\} \times M$. Moreover, the Liouville vector field is tangent to it and can be identified with (a reparametrization of) the underlying Anosov flow. In other words, the Liouville dynamics `remembers' the Anosov flow along the skeleton. It is also worth nothing that $\R \times M \setminus \mathfrak{skel}$ is the disjoint union of the symplectizations of $\alpha_-$ and $\alpha_+$; however, those symplectizations are `glued together' along their negative ends!

Perhaps a strategy to solve Question~\ref{quest:orbit2} would be to relate the two skeleta of the AL structures induced by the two Anosov flows $\Phi_0$ and $\Phi_1$. An intermediate step would be the following

\begin{question} \label{quest:orbit3}
    Let $\Phi$ be an Anosov flow with an induced AL structure $(\omega, \lambda)$ on $V \cong \R \times M$. Assume that there exists a smooth compactly supported function $f :V \rightarrow \R$ such that the following hold:
    \begin{itemize}
        \item The skeleton of $\lambda' \coloneqq \lambda + df$ is a codimension-$1$, $\mathcal{C}^1$ submanifold of $\R \times M$,
        \item The Liouville vector field $Z'$ of $\lambda'$ restricted to its skeleton generates an Anosov flow $\Phi'$.
    \end{itemize}
    Are $\Phi$ and $\Phi'$ orbit equivalent? Are they related in any way?
\end{question}

            \subsection{Towards a dynamical characterization}

We saw that the skeleton of an AL structure supporting an Anosov flow $\Phi$ is always a $\mathcal{C}^1$ codimension-$1$ submanifold of $V$ which is invariant under the Liouville flow. Hozoori further observed in~\cite{Hoz24-reg} that the skeleton is \textbf{$\mathcal{C}^1$-persistent}, in the sense that the skeleton of a $\mathcal{C}^1$-small perturbation of the Liouville vector field (not necessarily through Liouville vector fields!) is a $\mathcal{C}^1$ submanifold which is $\mathcal{C}^1$-close to the original skeleton. By deep theorems of Hirsh--Pugh--Shub and Ma\~{n}\'{e}, this is equivalent to the fact that the skeleton is \emph{normally hyperbolic} with respect to the Liouville flow. Interestingly, the corresponding statement does not hold for the skeleton of a transverse Liouville pair induced by a semi-Anosov but not Anosov flow. We refer to~\cite{Hoz24-reg} for a more complete discussion.

It is natural to conjecture that these properties of the skeleton are enough to \emph{characterize} AL manifolds:

\begin{conj} \label{conj:skelpersist}
    Let $(V, \omega, \lambda)$ be a $4$-dimensional finite type Liouville manifold satisfying the following assumptions.
    \begin{itemize}
        \item The skeleton $\mathfrak{skel}$ of $\lambda$ is a codimension-$1$ $\mathcal{C}^1$ submanifold,
        \item $\mathfrak{skel}$ is normally hyperbolic with respect to the Liouville flow,
        \item The Liouville flow induces an Anosov flow $\Phi$ on $\mathfrak{skel}$.
    \end{itemize}
    Then $(V, \omega, \lambda)$ is exact symplectomorphic to an AL structure supporting $\Phi$.
\end{conj}

A first step is obtained in~\cite[Theorem 1.12]{Hoz24-reg}: under these assumptions, the Liouville structure $\lambda$ is \emph{$\mathcal{C}^1$-strictly equivalent} to an AL structure, in the sense that there exists a $\mathcal{C}^1$-diffeomorphism $\varphi$ such that $\varphi^*\lambda$ is an AL structure supporting $\Phi$. Unfortunately, this does not immediately imply the previous conjecture, but the methods of~\cite{Hoz24-reg} could perhaps be extended to yield such a result.

\section{AL structures revisited} \label{sec:gAL}

In this section, we define a larger space of Liouville structures associated with an Anosov or semi-Anosov flow. The goal is twofold: we want to make minimal assumptions on such structures while ensuring that the space at hand is contractible for `essentially standard' reasons. 

    \subsection{Generalized AL structures}

Let $\Phi$ be an (oriented) Anosov flow on $M$, or more generally a semi-Anosov flow.

\begin{defn} \label{def:gAL}
    A \textbf{generalized Anosov Liouville structure} (or \textbf{gAL structure} for short) supporting $\Phi$ is a Liouville structure $(\omega, \lambda)$ on $[-1,1] \times M$ satisfying the following conditions.
    \begin{enumerate}
        \item $(\omega,\lambda)$ is a Liouville filling of a bicontact structure $(\xi_-, \xi_+)$ supporting $\Phi$.
        \item There exists a $u$-defining $1$-form $\alpha_u$ as in Definition~\ref{def:defpair} such that
        \begin{align} \label{eq:alphau}
        \omega \wedge d\alpha_u > 0.
    \end{align}
    \end{enumerate}

    If moreover the Reeb vector fields $R_\pm$ induced by $\lambda$ on $\{\pm1\}\times M$ satisfy that $R_+$ (resp.~$R_-$) is positively (resp.~negatively) transverse to $E^{ws}$, we say that $(\omega, \lambda)$ is \textbf{Reeb-adapted}.
\end{defn}

The $1$-form $\alpha_u$ is not part of the data of this definition. Notice that the space of $u$-defining $1$-forms satisfying~\eqref{eq:alphau} is convex, hence contractible. Interestingly, this second condition only depends on the symplectic form $\omega$, \emph{not on the Liouville form $\lambda$}! We shall refer to this condition as the \textbf{$u$-taming condition}.

\begin{rem}
    The $u$-taming condition can be rephrased in terms of the vector field $X_u$ on $[-1,1] \times M$ defined by $$\iota_{X_u} \omega = \alpha_u.$$ An easy computation shows that~\eqref{eq:alphau} is equivalent to $$\mathcal{L}_{X_u} (\omega \wedge \omega )> 0,$$
which means that $X_u$ has \emph{positive divergence} for the volume form $\omega \wedge \omega$.

This condition can also be rephrased by considering the (oriented) kernel of $d\alpha_u$ on $TM$, which is spanned by a vector $Z_u$ which projects to a positive generator of $\overline{E}^{ws} \subset TM \slash \langle X \rangle$. Equation~\eqref{eq:alphau} is then equivalent to 
$$\omega(Z_u, \partial_t) > 0.$$
Notice that $\mathrm{span}\{Z_u, \partial_t\}$ is tangent to a $2$-dimensional foliation on $[-1, 1] \times M$, which is then symplectic for $\omega$.

Let us give a last interpretation of condition~\ref{eq:alphau}. A symplectic form $\omega$ on $[-1,1] \times M$ is equivalent to the data of a family $(\alpha_t)_{t \in [-1,1]}$ of $1$-forms on $M$ and a family $(\omega_t)_{t \in [-1,1]}$ of \emph{closed} $2$-forms on $M$ satisfying that for every $t \in [-1,1]$, 
\begin{itemize}
    \item $\alpha_t \wedge \omega_t > 0$,
    \item $d \alpha_t = \partial_t \omega_t$.
\end{itemize}
It is easy to see that there is a unique $2$-form $\omega$ on $[-1,1] \times M$ satisfying $\iota_{\partial_t} \omega = \alpha_t$ and $\omega_t = \omega_{\vert \{t\} \times M}$. The first item is equivalent to $\omega$ being nondegenerate, and the second item is equivalent to $\omega$ being closed. The first item also means that $\alpha_t$ evaluates positively on $\ker \omega_t$, and $\omega_t$ evaluates positively on $\ker \alpha_t$. From this viewpoint, condition~\ref{eq:alphau} becomes 
$$\alpha_t \wedge d\alpha_u > 0,$$
i.e., $d\alpha_u$ is positive on $\xi_t \coloneqq \ker \alpha_t = \ker \iota_{\partial_t} \omega$. In that regards, condition~\ref{eq:alphau} means that the plane fields $\xi_t$ remain in the cone defined by $\{d\alpha_u > 0\}$; informally, they don't `twist'.
\end{rem}

We denote by $g\mathcal{AL}_\Phi$ the space of gAL structures on $[-1,1] \times M$ supporting $\Phi$, and by $g\mathcal{AL}^*_\Phi \subset g\mathcal{AL}_\Phi$ the subspace of Reeb-adapted ones. Notice that Definition~\ref{def:gAL} still makes sense when $\lambda$ is only $\mathcal{C}^1$.

\begin{ex} \label{ex:interp}
    The second condition in Definition~\ref{def:gAL} is automatically satisfied for $1$-forms on $[-1,1] \times M$ obtained as suitable interpolations between two contact forms $\alpha_\pm$ whose kernels form a bicontact structure supporting $\Phi$. This includes linear and exponential Liouville pairs (or rather truncations thereof). More generally, let us consider two such contact forms $\alpha_\pm$, and two functions $\chi_\pm : [-1,1]\times M \rightarrow \R$ satisfying
    \begin{align*}
        \pm \partial_t \chi_\pm > 0.
    \end{align*}
    We define a $1$-form $\lambda = \lambda_{(\chi_\pm, \alpha_\pm)}$ on $[-1,1] \times M$ by
    \begin{align*}
        \lambda \coloneqq \chi_- \alpha_- + \chi_+ \alpha_+.
    \end{align*}
    Such interpolations are extensively studied by Hozoori in~\cite{Hoz24-reg}. An immediate computation shows
    $$d\lambda \wedge d\alpha_u = dt \wedge \left( \partial_t \chi_t \, \alpha_- \wedge d\alpha_u + \partial_t\chi_+ \,  \alpha_+ \wedge d\alpha_u \right) > 0,$$
    since $\pm \alpha_\pm \wedge d\alpha_u > 0$. Of course, the first condition in Definition~\ref{def:gAL} imposes further highly nontrivial constraints on $\chi_\pm$ and $\alpha_\pm$.
\end{ex}

\begin{rem}
    Every bicontact structure supporting $\Phi$ admits a Liouville filling which is a $g$AL structure, since it admits a Liouville filling induced by a linear Liouville pair by~\cite[Theorem 1.6]{Hoz24}.
\end{rem}

We now state the main theorem of this section. 

\begin{thm} \label{thm:gcontract}
    The space of $g$AL structures supporting $\Phi$ is contractible. 
\end{thm}

We split into two steps summarized in the following technical lemmas, whose proofs are deferred to the next section.

\begin{lem} \label{lem:gsurjpin}
    The inclusion $g\mathcal{AL}^*_\Phi \subset g\mathcal{AL}_\Phi$ induces surjections on all homotopy groups.
\end{lem}

\begin{lem} \label{lem:gcontract*}
    The space $g\mathcal{AL}^*_\Phi$ is weakly contractible.
\end{lem} 

\begin{proof}[Proof of Theorem~\ref{thm:gcontract}]
    Lemma~\ref{lem:gsurjpin} and Lemma~\ref{lem:gcontract*} imply that $g\mathcal{AL}_\Phi$ is weakly contractible. Moreover, it is an open subset of a Fr\'{e}chet space, the space of smooth $1$-forms $\lambda$ on $[-1,1] \times M$ such that the flow $\Phi$ is contained in $\ker \lambda$ on $\{\pm1\} \times M$ (a closed subspace of a Fr\'{e}chet space), hence is homotopy equivalent to a CW complex (see~\cite[Section 4]{Mas23} for a complete argument). Whitehead theorem then implies that $g\mathcal{AL}_\Phi$ is contractible.
\end{proof}

With Example~\ref{ex:interp} in mind, Theorem~\ref{thm:gcontract} readily implies:

\begin{cor}
    Any two (linear or exponential) Liouville pairs supporting a given (semi-)Anosov flow give rise to equivalent Liouville structures on $[-1,1]\times M$.
\end{cor}

While the $u$-taming condition of Definition~\ref{def:gAL} is crucial in our proof, we ask:

\begin{question}
If $\big([-1,1] \times M, \lambda \big)$ is a Liouville filling of a bicontact structure $(\xi_-, \xi_+)$ supporting an Anosov (or semi-Anosov) flow $\Phi$, is it Liouville homotopic to a $g$AL structure supporting $\Phi$?
\end{question}

Note that a given pair of oppositely oriented contact structures $(\xi_-, \xi_+)$ might admit different Liouville fillings of the form $[-1,1] \times M$ which are homotopic to $g$AL structures supporting \emph{different} Anosov flows. For instance, $(\xi_-, \xi_+)$ might be a bicontact structure supporting an Anosov flow $\Phi$ such that there exists another contact structure $\xi'_-$, contact homotopic to $\xi_-$, such that $(\xi'_-, \xi_+)$ is a bicontact structure supporting another Anosov flow $\Phi'$. In this situation, we do not know if $\Phi$ and $\Phi'$ are orbit-equivalent, see Question~\ref{quest:orbit4} above. However, if $\Phi$ and $\Phi'$ are orbit equivalent, then Theorem~\ref{thm:orbiteq} above implies that the $g$AL structures supporting $\Phi$ and $\Phi'$ are all exact symplectomorphic to each other.

On the other end, an affirmative answer to this question would imply that the following one also holds:

\begin{question}
    Let $\Phi$, $\Phi'$ be two Anosov flows on $M$ supporting bicontact structures $(\xi_-, \xi_+)$ and $(\xi'_-, \xi'_+)$, respectively, such that $\xi_\pm$ and $\xi'_\pm$ are contact homotopic. Are the $g$AL structures supporting $\Phi$ and $\Phi'$ Liouville homotopic to each other?
\end{question}

\begin{rem}
    There exists many closed $3$-manifolds $M$ such that $[-1,1] \times M$ admits a Liouville structure, but $M$ does not admit an Anosov flow. Liouville structures on $[-1,1] \times M$ can be constructed from suitable taut foliations that we call \emph{hypertaut} in~\cite{Mas24}. Taut foliations on rational homology spheres are automatically hypertaut.\footnote{Here, `foliation' means $\mathcal{C}^0$-foliation and `taut' means everywhere taut, see~\cite{CKR19}.} For instance, Bin Yu showed in~\cite{Yu23} that every non-integral surgery on the figure-eight knot in $S^3$ yields a rational homology sphere which does not admit Anosov flows. However, every nontrivial surgery on the figure-eight knot admits a taut foliation by the work of Gabai.
\end{rem}

        \subsection{Proofs of technical lemmas}

We now turn to the proofs of the two lemmas used in the proof of Theorem~\ref{thm:gcontract}.

\begin{proof}[Proof of Lemma~\ref{lem:gsurjpin}]
Let $(\alpha^1_-, \alpha^1_+)$ be a pair of contact forms defining a bicontact structure supporting $\Phi$ and whose Reeb vector fields $R^1_\pm$ are transverse to $E^{ws}$ with the correct signs, as in the definition of Reeb-adaptedness.

We first show that every $(\omega, \lambda) \in g\mathcal{AL}_\Phi$ is homotopic within $g\mathcal{AL}_\Phi$ to a $g$AL structure $(\omega_1, \lambda_1)$ whose boundary contact \emph{forms} are $\alpha^1_\pm$, which guarantees Reed-adaptedness. This shows that the map $\pi_0\left( g\mathcal{AL}^*_\Phi \right) \rightarrow \pi_0\left( g\mathcal{AL}_\Phi \right)$ induced by the inclusion is surjective. We later explain how to adapt the proof to show surjectivity on $\pi_n$ for $n \geq 1$.

We deform $\lambda$ near $\{\pm1\} \times M$ in two steps.
\begin{itemize}[leftmargin=*]
    \item \textit{Step 1: straightening the Liouville vector field near the boundary.} Since $\lambda$ is a Liouville form, its Liouville vector field $Z$ is positively transverse to the boundary of $[-1,1] \times M$. We deform $\lambda$ near this boundary so that the new Liouville vector field coincides with $t \partial_t$ near $\{\pm1\} \times M$. Some care is needed to ensure that this deformation stays within $g$AL structures, but the strategy is standard.

    Let us deform $\lambda$ near $\{1\} \times M$, the deformation near $\{-1\} \times M$ being similar. Writing $\alpha_+ = \lambda_{\left\vert \{1\} \times M \right.}$, there exist $\epsilon > 0$ and a $1$-form $\mu$ such that 
    $$\lambda = t\alpha_+ + \mu$$
    on $N_\epsilon \coloneqq (1-\epsilon, 1] \times M$, and $\mu$ vanishes on $\{1\} \times M$. We can further write $$\mu = (t-1) \delta$$ for some $1$-form $\delta$ on $N_\epsilon$, and the condition that $\lambda$ is symplectic near $\{1\} \times M$ is equivalent to
    \begin{align*}
        (\alpha_+ + \delta_+) \wedge d\alpha_+ > 0,
    \end{align*}
    where $\delta_+ \coloneqq \delta_{\left\vert \{1\} \times M\right.}$. Let us define a smooth function $f_+ : M \rightarrow \R_{>0}$ by
    $$(\alpha_+ + \delta_+) \wedge d\alpha_+ = f_+ \, \alpha_+ \wedge d\alpha_+.$$

    Similarly, the $u$-taming condition near $\{1\} \times M$ is equivalent to
    \begin{align*}
        (\alpha_+ + \delta_+) \wedge d\alpha_u > 0,
    \end{align*}
    and we define a continuous function $f_u : M \rightarrow \R_{>0}$ by
    $$(\alpha_+ + \delta_+) \wedge d\alpha_u = f_u \, \alpha_+ \wedge d\alpha_+.$$
    
    We set
    \begin{align*}
        m_0 &\coloneqq \mathrm{min}\big\{ \mathrm{min}(f_+), \mathrm{min}(f_u), 1/2\big\}, \\
        \kappa_0 &\coloneqq \frac{1}{1-m_0} > 0.
    \end{align*}
    Let $0 < \epsilon_0 < \epsilon$ to be chosen sufficiently small later, and let $\varphi : (1-\epsilon, 1] \rightarrow (-\epsilon, 0]$ be a smooth function satisfying:
    \begin{enumerate}
        \item $\forall t \in (1-\epsilon, 1-\epsilon_0]$, $\varphi(t) = t -1 $,
        \item $\forall t \in (1-\epsilon, 1]$, $0 \leq \varphi'(t) \leq 1 + \kappa_0$,
        \item $\forall t \in [1-\frac{\epsilon_0}{2}, 1]$, $\varphi(t) = 0$.
    \end{enumerate}

We define a $1$-form $\widetilde{\lambda}_+$ on $N_\epsilon$ by
$$\widetilde{\lambda}_+ \coloneqq t \alpha_+ + \varphi(t) \delta,$$
so that $\lambda$ and $\widetilde{\lambda}_+$ coincide on $N_\epsilon \setminus N_{\epsilon_0}$. Moreover, $\widetilde{\lambda}_+$ coincides with $ t \alpha_+$ on $N_{\frac{\epsilon_0}{2}}$, which is a Liouville form with Liouville vector field $t\partial_t$. We claim that for $\epsilon_0$ small enough, $\widetilde{\omega}_+ \coloneqq d\widetilde{\lambda}_+$ is symplectic and satisfies $\widetilde{\omega}_+ \wedge d\alpha_u > 0$. It is enough to show these properties on $N_{\epsilon_0}$.

A straightforward computation shows that
\begin{align*}
    \widetilde{\omega}_+ \wedge \widetilde{\omega}_+ &= 2 tdt \wedge \big( \alpha_+ + \varphi' \delta_+\big)\wedge d\alpha_+ + \varphi \, \Theta_+,
\end{align*}
where $\Theta_+$ is a $4$-form depending on $\varphi$ which is bounded independently of $\varphi$ and $\epsilon_0$. More precisely, writing $\mathrm{dvol} = dt \wedge \alpha_+ \wedge d\alpha_+$ and $\Theta_+ = \theta_+ \, \mathrm{dvol}$, the function $\vert \theta_+ \vert$ is bounded by a constant $C_+$ independent of $\varphi$ and $\epsilon_0$. The previous equation becomes
$$\widetilde{\omega}_+ \wedge \widetilde{\omega}_+ = \left\{2t\big( 1 + \varphi' (f_+-1) \big) + \varphi \theta_+ \right\} \mathrm{dvol},$$
and our choice of $\varphi$ ensures that $t\big(1 + \varphi' (f_+ -1)\big) > \epsilon_+ > 0$ for some $\epsilon_+$ independent of $\epsilon_0$. Notice that on $(1-\epsilon_0, 1]$, $-\epsilon_0 < \varphi \leq 0$ so choosing $\epsilon_0$ such that $C_+ \epsilon_0 \ll \epsilon_+$ ensures that $\widetilde{\omega}_+ \wedge \widetilde{\omega}_+ > 0$ on $N_{\epsilon_0}$, as desired.

Similarly, we have
\begin{align*}
    \widetilde{\omega}_+ \wedge d\alpha_u = \left\{t\big( 1 + \varphi' (f_u-1) \big) + \varphi \theta_u\right\} \mathrm{dvol},
\end{align*}
where $\theta_u$ is a (continuous) function such that $\vert \theta_u \vert \leq C_u$ for a constant $C_u$ independent of $\varphi$ and $\epsilon_0$. Moreover, $t\big(1 + \varphi' (f_u -1)\big) > \epsilon_u > 0$ for some $\epsilon_u$ independent of $\epsilon_0$. Choosing $\epsilon_0$ so that $C_u \epsilon_0 \ll \epsilon_u$ guarantees that $\widetilde{\omega}_+ \wedge d\alpha_u > 0$ on $N_{\epsilon_0}$.

By symmetry, we can also construct a $1$-form $\widetilde{\lambda}$ on $[-1,-1+\epsilon) \times M$ which coincides with $\lambda$ on $[-1+\epsilon_0, -1+ \epsilon) \times M$ and with $t \alpha_- = t\lambda_{\left\vert \{-1\} \times M \right.}$ on $[-1, -1+ \frac{\epsilon_0}{2}]\times M$, and which satisfies that $\widetilde{\omega}_- \coloneqq d\widetilde{\lambda}_-$ is symplectic and $\widetilde{\omega}_- \wedge d\alpha_u > 0$. We then define a $1$-form $\widetilde{\lambda}$ on $[-1,1] \times M$ as 
$$\widetilde{\lambda} \coloneqq \begin{cases} \widetilde{\lambda}_- &\mathrm{on} \ [-1, -1+\epsilon) \times M,\\
\lambda &\mathrm{on}\ [-1+\epsilon, 1-\epsilon] \times M,\\
\widetilde{\lambda}_+ & \mathrm{on}\ (1-\epsilon, 1] \times M,
\end{cases}$$
which defines a $g$AL structure supporting $\Phi$ and whose Liouville vector field is $t\partial_t$ near the boundary of $[-1,1]\times M$.

Finally, we claim that the linear interpolation $\lambda_\tau \coloneqq (1-\tau) \lambda + \tau \widetilde{\lambda}$, $0 \leq \tau \leq 1$, induces a path of $g$AL structures supporting $\Phi$ (after possibly shrinking $\epsilon_0$ further). Obviously, $d\lambda_\tau \wedge d\alpha_u > 0$, and $\lambda_\tau$ restricts to $\alpha_\pm$ on $\{\pm1\} \times M$, so it suffices to show that $\omega_\tau \coloneqq d\lambda_\tau$ is symplectic. It is obvious on $[-1+\epsilon, 1-\epsilon] \times M$, and we show it on $N_\epsilon = (1-\epsilon, 1] \times M$, the case of $[-1, -1+\epsilon) \times M$ being identical.

On $N_\epsilon$, $\widetilde{\lambda}$ writes
$$\widetilde{\lambda} = t\alpha_+ + \varphi_\tau(t) \delta,$$
where $\varphi_\tau(t) = (1-\tau) (t-1) + \tau \varphi(t)$. Notice that $\varphi_\tau$ satisfies the conditions 1 and 2 in the definition of $\varphi$, and $\widetilde{\omega} \wedge \widetilde{\omega}$ is of the form
$$\widetilde{\omega} \wedge \widetilde{\omega} = 2tdt \wedge \big( \alpha_+ + \varphi'_\tau \delta_+\big)\wedge d\alpha_+ + \varphi_\tau \, \Theta_\tau,$$
where $\Theta_\tau$ is a $4$-form which is bounded independently of $\varphi$, $\tau$, and $\epsilon_0$. The same strategy as before guarantees that $\widetilde{\omega} \wedge \widetilde{\omega} > 0$ for $\epsilon_0$ small enough. This completes Step 1. Notice that the boundary contact structures remain unchanged during that step.

\item \textit{Step 2: modifying the boundary contact forms.} By Step 1, we can now assume that near $\{\pm1\} \times M$, $\lambda$ is of the form $t \alpha_\pm$. Choosing $\epsilon > 0$ so that the latter is satisfied on $\big([-1,-1+2\epsilon) \cup (1-2\epsilon, 1] \big) \times M$, we modify $\lambda$ in this region into a new Liouville form $\lambda_1$ inducing a $g$AL structure whose induced contact forms at the boundary components are $C \alpha^1_\pm$ for some constant $C > 0$. In particular, $\lambda_1$ is Reeb-adapted. Here as well, the strategy is quite standard.

Since the space of bicontact structures supporting $\Phi$ is contractible (see~\cite[Theorem 4.4]{Mas23}), we can find a path of contact forms $(\alpha^\kappa_-, \alpha^\kappa_+)$, $0 \leq \kappa \leq 1$, such that $\alpha^0_\pm = \alpha_\pm$ and whose kernels define bicontact structures supporting $\Phi$. We can further assume that this path is constant near $\kappa \in \{0,1\}$. Let us describe the modification on $N_{2\epsilon} = (1-2\epsilon ,1] \times M$, the other case being identical.

Let us define functions $f_+, f_u : [0,1]_\kappa\times M \rightarrow \R_{>0}$ and $g_+, g_u : [0,1]_\kappa \times M \rightarrow \R$ by
\begin{align*}
    \alpha^\kappa_+ \wedge d\alpha^\kappa_+ &= f_+(\kappa, \cdot \ ) \, \mathrm{dvol}_M, & \partial_\kappa \alpha^\kappa_+ \wedge d\alpha^\kappa_+ &= g_+(\kappa, \cdot \ ) \, \mathrm{dvol}_M,\\
    \alpha^\kappa_+ \wedge d\alpha_u &= f_u(\kappa, \cdot \ ) \, \mathrm{dvol}_M,& \partial_\kappa\alpha^\kappa_+ \wedge d\alpha_u &= g_u(\kappa, \cdot \ ) \, \mathrm{dvol}_M,
\end{align*}
where $\mathrm{dvol}_M \coloneqq \alpha_+ \wedge d\alpha_+$. We choose $A > 0$ satisfying
\begin{align} \label{eq:A}
        \epsilon A > \mathrm{max}\left\{ \frac{ \max\vert g_+ \vert}{\min f_+}, \frac{\vert \max g_u \vert }{\min f_u}\right\}.
\end{align}

Let $\widehat{\alpha}_+$ be the $1$-form on $N_{2\epsilon}$ defined by
$$\widehat{\alpha}_+(t,x) \coloneqq \begin{cases}
    \alpha_+(x) &\mathrm{if \ } t \in [1-2\epsilon , 1-\epsilon], \\
    \alpha^{(t-1+\epsilon)/\epsilon}_+(x) &\mathrm{if \ } t \in [1-\epsilon, 1].
\end{cases}$$
We further choose a smooth function $v : [1-2\epsilon, 1] \rightarrow \R$ satisfying
\begin{enumerate}
    \item $\forall t \in [1-2\epsilon, 1],$ $v'(t) > 0$,
    \item $\forall t \in [1-2\epsilon, 1 - \frac{3\epsilon}{2}],$ $v(t) = \ln(t)$,
    \item $\forall t \in [1-\epsilon, 1],$ $v(t) = A(t-1+\epsilon)$,
\end{enumerate}
and we define
$$\widehat{\lambda}_+ \coloneqq e^v \widehat{\alpha}_+$$
on $N_{2\epsilon}$.
Notice that $\widehat{\lambda}_+$ coincides with $\lambda$ on $N_{2\epsilon} \setminus N_{\frac{3\epsilon}{2}}$. We check that $\widehat{\omega}_+ \coloneqq d\widehat{\lambda}$ is symplectic and satisfies $\widehat{\omega}_+ \wedge d\alpha_u > 0$. It is automatic on $N_{2\epsilon} \setminus N_\epsilon$ by our choice of $v$. We then check it on $N_\epsilon$.

A straightforward computation shows that at a point $ p = (t,x)  \in N_\epsilon$,
$$\widehat{\omega}_+ \wedge \widehat{\omega}_+(p) = 2e^{2v(t)} \left( A f_+ + \frac{g_+}{\epsilon}\right)\!\big(\tau, x\big) \, dt \wedge \mathrm{dvol}_M > 0,$$
where $\tau = \frac{t-1+\epsilon}{\epsilon}$,
by our choice of $v$ and $A$ as in~\eqref{eq:A}, and
$$\widehat{\omega}_+ \wedge d\alpha_u(p) =  e^v\left( A f_u + \frac{g_u}{\epsilon} \right)\!(\tau,x) \,  dt\wedge \mathrm{dvol}_M > 0$$
as well. Notice that $\widehat{\lambda}_+$ restricts to $C \alpha^1_+$ on $\{1\} \times M$, for some constant $C > 0$.

We similarly construct a $1$-form $\widehat{\lambda}_-$ on $[-1, -1 + 2\epsilon) \times M$, and we define
$$\widehat{\lambda} \coloneqq \begin{cases} \widehat{\lambda}_- &\mathrm{on} \ [-1, -1+2\epsilon] \times M,\\
\lambda &\mathrm{on}\ [-1+2\epsilon, 1-2\epsilon] \times M,\\
\widehat{\lambda}_+ & \mathrm{on}\ [1-2\epsilon, 1] \times M,
\end{cases}$$
which induces a Reeb-adapted $g$AL structure supporting $\Phi$.

It is then easy to show that $(\omega,\lambda)$ and $\big(\widehat{\omega}, \widehat{\lambda}\big)$ are homotopic through $g$AL structures supporting $\Phi$. For $\tau\in [0,1]$, we define $\psi_\tau : [-1,1] \rightarrow [-1,1]$ by $\psi_\tau(t) \coloneqq (1-2\tau \epsilon)t$, and $\Psi_\tau : [-1,1] \times M \rightarrow [-1,1] \times M$ as $\psi_\tau \times \mathrm{id}_M$. Then, the path of $1$-forms 
$$\lambda_\tau \coloneqq 
\begin{cases}
    \Psi^*_{2\tau}\lambda &\mathrm{for \ } \tau \in [0, \frac{1}{2}], \\
    \Psi^*_{2(1-\tau)}\widehat{\lambda} & \mathrm{for \ } \tau \in [\frac{1}{2}, 1]
\end{cases}$$
provides the desired homotopy.
\end{itemize}

These two steps show that $(\omega, \lambda)$ is homotopic to a Reeb-adapted $g$AL structure  through $g$AL structures supporting $\Phi$. Let us now consider a family $(\omega_\sigma,\lambda_\sigma) \in g\mathcal{AL}_\Phi$, $\sigma \in S^n$, indexed over the $n$-sphere. We want to construct a homotopy $[0,1]\times S^n \rightarrow g\mathcal{AL}_\Phi$ from the chosen one on $\{0\} \times S^n$ to a family on $\{1\} \times S^n$ which belongs to $g \mathcal{AL}^*_\Phi$. For that, we first choose a (continuous) family of $u$-defining $1$-forms $\alpha^\sigma_u$ satisfying $\omega_\sigma \wedge d\alpha^\sigma_u > 0$; this is possible since for each fixed $\sigma$, the space of such $\alpha^\sigma_u$ is open and convex. We then apply the two previous steps to $(\omega_\sigma, \lambda_\sigma)$, by choosing the relevant parameters uniformly over $S^n$. Notice that we can impose the final boundary contact forms to be independent of $\sigma \in S^n$. The details are left to the reader.
\end{proof}

\begin{proof}[Proof of Lemma~\ref{lem:gcontract*}]
    Let us fixed a base point $(\omega_0, \lambda_0) \in g\mathcal{AL}^*_\Phi$ given by the standard linear construction of Section~\ref{sec:Mit}. We first show that every $(\omega, \lambda) \in g\mathcal{AL}^*_\Phi$ is homotopic to $(\omega_0, \lambda_0)$ within $g\mathcal{AL}^*_\Phi$. We then explain how to adapt the proof to show that $g\mathcal{AL}^*_\Phi$ is (weakly) contractible.

    Let $(\omega, \lambda) \in g\mathcal{AL}^*_\Phi$ be arbitrary, and let $\alpha_u$ be a $u$-defining $1$-form satisfying~\eqref{eq:alphau}. Notice that the Reeb-adaptedness condition is equivalent to 
    \begin{align} \label{eq:reebadapt}
        \pm \alpha_u \wedge d\alpha_\pm > 0. 
    \end{align}

    We claim that for every $\kappa > 0$, the $\mathcal{C}^1$ $1$-form
    $$\lambda^\kappa \coloneqq \lambda + \kappa \alpha_u$$
    defines a Reeb-adapted $g$AL structure supporting $\Phi$ of class $\mathcal{C}^1$. We need to check the following.

    \begin{itemize}
        \item \textit{$\omega^\kappa \coloneqq d\lambda^\kappa$ is symplectic and satisfies~\eqref{eq:alphau}.} We simply compute
            $$\omega^\kappa \wedge \omega^\kappa = \omega \wedge \omega + 2\kappa \, \omega \wedge d\alpha_u > 0$$
        by \eqref{eq:alphau}. Moreover, $\omega^\kappa \wedge d\alpha_u = \omega \wedge d\alpha_u> 0$ as desired.
        
        \item \textit{$\lambda^\kappa$ restricts to contact forms on $\{\pm1\} \times M$.} The restriction of $\lambda^\kappa$ to $\{\pm1\} \times M$ is simply $\alpha^\kappa_\pm \coloneqq \alpha_\pm + \kappa \alpha_u$. Since $\alpha_u \wedge d \alpha_u = 0$, we have:
        $$\alpha^\kappa_\pm \wedge d\alpha^\kappa_\pm = \alpha_\pm \wedge d\alpha_\pm + \kappa \langle \alpha_\pm, \alpha_u\rangle,$$
        where $\langle \alpha_\pm, \alpha_u\rangle \coloneqq \alpha_\pm \wedge d\alpha_u + \alpha_u \wedge d\alpha_\pm$. Notice that
        $$\alpha_\pm \wedge d\alpha_u (X, e_s, e_u) = \alpha_\pm(e_s) d\alpha_u(e_u, X) = -r_u \alpha_\pm(e_s)$$
        where $r_u >0$ and $\mp \alpha_\pm(e_s) >0$, hence
        $$\pm \alpha_\pm \wedge d\alpha_u > 0.$$
        Together with~\eqref{eq:reebadapt}, this implies that $\pm \langle \alpha_\pm, \alpha_u \rangle > 0$ hence $\pm \alpha^\kappa_\pm \wedge d\alpha^\kappa_\pm > 0$.
        
        \item \textit{$\lambda^\kappa$ is Reeb-adapted.} We trivially have $\pm \alpha_u \wedge d\alpha^\kappa_\pm = \pm \alpha_u \wedge d\alpha_\pm > 0$ by~\eqref{eq:reebadapt}.
    \end{itemize}

    We similarly define
    $$\lambda^\kappa_0 \coloneqq \lambda_0 + \kappa \alpha_u,$$
    which is also a Reeb-adapted $g$AL structure supporting $\Phi$ of class $\mathcal{C}^1$.

    We now claim that for $\kappa$ large enough and for every $\tau \in [0,1]$, 
    \begin{align*}
    \lambda^\kappa_\tau &\coloneqq (1-\tau) \lambda^\kappa_0 + \tau \lambda^\kappa \\
    &= (1-\tau) \lambda_0 + \tau \lambda + \kappa \alpha_u
    \end{align*}
    is a Reeb-adapted $g$AL supporting $\Phi$ of class $\mathcal{C}^1$. We adapt the previous steps.

    \begin{itemize}
        \item \textit{$\omega^\kappa_\tau \coloneqq d\lambda^\kappa_\tau$ is symplectic and satisfies~\eqref{eq:alphau}.} We compute
            \begin{align*}
                \omega^\kappa_\tau \wedge \omega^\kappa_\tau &= 2\left( (1-\tau) \omega_0 + \tau \omega \right)^{\wedge2} + 2 \kappa \left((1-\tau) \omega_0 + \tau \omega \right) \wedge d\alpha_u,
            \end{align*}
        which is positive for $\kappa$ large enough (and independent of $\tau$). The inequality $\omega^\kappa_\tau \wedge d\alpha_u > 0$ is obviously satisfied.
        
        \item \textit{$\lambda^\kappa_\tau$ restricts to contact forms on $\{\pm1\} \times M$.} Denote by $\alpha^{\kappa, \tau}_\pm$ the restriction of $\lambda^\kappa_\tau$ to $\{\pm1\} \times M$. Then $\alpha^{\kappa, \tau}_\pm \wedge d\alpha^{\kappa, \tau}_\pm$ is of the form
        $$\alpha^{\kappa, \tau}_\pm \wedge d\alpha^{\kappa, \tau}_\pm = \beta_\tau + \kappa\left((1-\tau) \langle \alpha^0_\pm, \alpha_u \rangle  + \tau \langle \alpha_\pm, \alpha_u\rangle \right),$$
        where $\beta_\tau$ is a $1$-form on $M$ independent of $\kappa$. We already showed earlier that $\pm \langle \alpha^0_\pm, \alpha_u \rangle > 0$ and $\pm \langle \alpha_\pm, \alpha_u \rangle > 0$, so $\pm \alpha^{\kappa, \tau}_\pm \wedge d\alpha^{\kappa, \tau}_\pm > 0$ for $\kappa$ large enough (independent of $\tau$).
        
        \item \textit{$\lambda^\kappa_\tau$ is Reeb-adapted.} We trivially have 
        $$\pm \alpha_u \wedge d\alpha^{\kappa,\tau}_\pm = \pm \alpha_u \wedge \left((1-\tau)d\alpha^0_\pm + \tau d\alpha_\pm\right)> 0$$
        by~\eqref{eq:reebadapt}.
    \end{itemize}

    To summarize, we obtain a path from $\lambda_0$ to $\lambda$ within $\mathcal{C}^1$ Reeb-adapted $g$AL structures supporting $\Phi$ by choosing $\kappa >0 $ large enough (depending on $\lambda_0$ and $\lambda$), and concatenating a linear interpolation from $\lambda_0$ to $\lambda^\kappa_0$, a linear interpolation from $\lambda^\kappa_0$ to $\lambda^\kappa$, and a linear interpolation from $\lambda^\kappa$ to $\lambda$. This path can easily be upgraded to a path of \emph{smooth} $g$AL structures in $g\mathcal{AL}^*_\Phi$ from $(\omega_0, \lambda_0)$ to $(\omega, \lambda)$ (one could either smooth $\lambda$ directly, or consider a suitable smoothing of $\alpha_u$ instead).

    Let us now consider a family $(\omega_\sigma,\lambda_\sigma) \in g\mathcal{AL}^*_\Phi$, $\sigma \in S^n$, indexed over the $n$-sphere. We want to extend it to the $(n+1)$-ball $D^{n+1}$. As before, we first choose a (continuous) family of defining $1$-forms $\alpha^\sigma_u$ satisfying $\omega_\sigma \wedge d\alpha^\sigma_u > 0$. Then, we choose a sufficiently large $\kappa >0$, independent on $\sigma$, such that for every $\sigma \in S^n$, the previous construction gives a path from $\lambda_0$ to $\lambda_\sigma$ within $g \mathcal{AL}^*_\Phi$. Here, we use that for every $\sigma \in S^n$, $\omega_0 \wedge d\alpha^\sigma_u > 0$. It can easily be upgraded to continuous map $D^{n+1} \rightarrow g \mathcal{AL}^*_\Phi$ restricting to the family $(\omega_\sigma, \lambda_\sigma)$, $\sigma \in S^n$, on $\partial D^{n+1}$. This shows that $g \mathcal{AL}^*_\Phi$ is weakly contractible.
\end{proof}

\printbibliography[heading=bibintoc, title={References}]
\end{document}